\documentclass[a4paper,reqno]{amsart}
\pdfoutput=1

\usepackage{xcolor}
\usepackage{enumerate} 
\usepackage{enumitem}
\usepackage{dsfont}
\usepackage{relsize}
\usepackage{amsmath,amsfonts,amssymb,amsthm,
stmaryrd,bbm,graphicx,mathtools,enumerate,euscript}

\usepackage{bbold}
\usepackage{bbm}
\usepackage{mathrsfs}

\usepackage{mathrsfs}
\usepackage{hyphenat}
\usepackage[T1]{fontenc} 
\usepackage[latin1]{inputenc}
\usepackage[english]{babel}
\usepackage[tracking,spacing,kerning,babel]{microtype}
\usepackage{wasysym} 
\usepackage{color}
\usepackage{xcolor}
    	\usepackage{framed} 
\usepackage{wrapfig} 
\newcommand{\E}{\mathbb{E}}

\newcommand{\R}{\mathbb{R}}

\newcommand{\Cb}{\mathbf{C}}

\usepackage[final]{hyperref}   
\hypersetup{
    linktoc=page,
    linkcolor=red,          
    citecolor=blue,        
    filecolor=blue,      
    urlcolor=cyan,
   colorlinks=true           
}

\usepackage{mathalfa}

\usepackage{upgreek}
\newtheoremstyle{mine}
{\baselineskip}
{\baselineskip}
{\itshape}
{
}
{\bfseries}
{.}
{.5em}
{#1 #2\ifx#3\relax\else~(#3)\fi}

\theoremstyle{mine}

\newtheorem{theorem}{Theorem}

\numberwithin{theorem}{section}

\newtheorem{proposition}[theorem]{Proposition}
\newtheorem{lemma}[theorem]{Lemma}

\newtheorem{definition}[theorem]{Definition}

\numberwithin{equation}{section}

\theoremstyle{remark}
\newtheorem{remark}{Remark}

\title{Locally finite Fixed points of Branching Brownian Motion}
\author{Xinxin Chen, Arnab Chowdhury, Atul Shekhar, Shuo Zhu}

\address[Xinxin Chen]
{Beijing Normal University, School of Mathematical Sciences, China}
\email{xinxin.chen@bnu.edu.cn}

\address[Arnab Chowdhury]
{Tata Institute of Fundamental Research-CAM, Bangalore, India}

\email{arnab2020@tifrbng.res.in}

\address[Atul Shekhar]
{Tata Institute of Fundamental Research-CAM, Bangalore, India}
\email{atul@tifrbng.res.in}

\address[Shuo Zhu]
{Beijing Normal University, School of Mathematical Sciences, China}
\email{zhushuo@mail.bnu.edu.cn}
\begin{document}

\begin{abstract}   

We give a complete characterization of the fixed points of branching Brownian
motion with critical or supercritical drift, without any assumption beyond
almost-sure local finiteness. In particular, we assume neither local finiteness
of the intensity measure (cf. \cite{Kabluchko}) nor the existence of an
almost-surely finite rightmost particle (cf. \cite{Shekhar_FP}). We also
characterize the domains of attraction of these fixed points.

\end{abstract}
\maketitle
\section{Introduction}\label{Section Intro}
 
 \subsection{The setup and the goal.}
 A binary branching Brownian motion (BBM) is defined as follows. Each particle moves as a standard Brownian motion, independently of the other particles, and branches into two offspring at rate \(1\). For a BBM started from a single
particle at the origin, let \(\{\chi_k(t)\}_{1\leq k\leq n(t)}\) denote the particle positions at time \(t\), where \(n(t)\) is the number of particles alive at that time. The corresponding BBM with drift parameter \(\lambda\), meaning particle drift \(-\lambda\), has particle positions \(\{\chi_k(t)-\lambda t\}_{1\leq k\leq n(t)}\). BBM as a Markov process, and in particular its invariant distributions, was studied in
\cite{Kabluchko}, \cite{Shekhar_FP}, and \cite{Shekhar_Domain}. Here we
revisit the classification of these invariant distributions under the sole
assumption that the random point measure is almost surely locally finite.\\
 
 Let $\mathcal{N}$ be the space of all locally finite point measures on $\mathbb{R}$, see section \ref{Section Laplace}. We follow the standard terminology: a point process $\theta$ is a $\mathcal{N}$-valued random variable, see [Chapter $2$-\cite{Bovier_book}]\footnote{Although this is a standard terminology, since we will also encounter random point measures which are possibly locally infinite, we stress that the term \textit{point process} will be used only for a $\mathcal{N}$-valued random variable. We will use ``random collection of points" to refer to random possibly locally infinite point measures.}. 
 
 The BBM with drift $\lambda \geq 0 $ started from a point process $\theta$ is informally defined as follows. Given $\theta$, for each atom $x_i$ of $\theta$, we run an independent BBM with drift $\lambda$ starting from $x_i$. The collection of points at time $t$ is given by  

 \begin{equation} \label{theta_t-def}
        \theta_t^\lambda:= \sum\limits_{i \in I} \sum\limits_{k=1}^{n^i(t)} \delta_{x_i+\chi_k^i(t)-\lambda t},
    \end{equation}
where $\theta_0^\lambda = \theta=\sum \limits_{i \in I} \delta_{x_i}$ and $\{\chi_k^i(t), 1 \le k \le n^i(t)\}_{i \in I}$ is a family of independent BBMs.  \\

For a general initial state in \(\mathcal N\), the evolution
\(\theta\mapsto\theta_t^\lambda\) is initially only formal: it need not define a
Markov process on \(\mathcal N\). The obstruction, which we call
\emph{coming down from infinity}, is that \(\theta_t^\lambda\) may fail to be
locally finite even when \(\theta\) is locally finite. For example, suppose that
\(\theta((x,x+1))\) grows on the order of \(e^{|x|^3}\) as
\(x\to\pm\infty\). Since the Brownian transition density decays on the Gaussian
scale, a Borel--Cantelli argument shows that \(\theta_t^\lambda\) is almost
surely locally infinite for every \(t>0\). Its law as a point process is then
undefined, so the usual invariance question is not well posed on all of
\(\mathcal N\).

One remedy is to restrict the evolution to a smaller state space, such as
\(\mathcal N_2\). Lemma \ref{Proposition M_2} shows that
\begin{equation}\label{N_2-is-fixed}
 \theta\in\mathcal N_2\ {\rm a.s.}
 \quad\Longrightarrow\quad
 \theta_t^\lambda\in\mathcal N_2\ {\rm a.s.}\quad\text{for every }t>0.
\end{equation}
This restriction, however, considers only invariant distributions supported on
\(\mathcal N_2\) and could therefore exclude invariant distributions outside
that space. To treat all almost surely locally finite initial point measures, we
instead adopt an approach analogous to that of Biskup and Louidor
\cite{Biskup-Louidor-GFF-local-extreme}. The results below will show
a posteriori that the two approaches are consistent.\\

Let $\theta$ be a point process and $\theta_t^\lambda$ be as given by \eqref{theta_t-def} which is possibly locally infinite. For $f\in C_c^+(\mathbb{R})$ (the space of compactly supported non-negative continuous functions on $\mathbb{R}$), let 

\begin{equation}
    \langle f, \theta_t^\lambda \rangle:= \int_{-\infty}^\infty f(x)\theta_t^\lambda (dx)=\sum\limits_{i \in I} \sum_{k=1}^{n^i(t)} f(x_i+\chi_k^i(t)-\lambda t).
\end{equation}

Although it can be possibly infinite, $\langle f, \theta_t^\lambda \rangle$ is a well-defined $[0, \infty]$-valued random variable. We can consider its Laplace transform $\mathbb{E}[e^{-\langle f, \theta_t^\lambda \rangle}]$ where $e^{-\langle f, \theta_t^\lambda \rangle} = 0 $ on the event $\{\langle f, \theta_t^\lambda \rangle = + \infty\}$.

\begin{definition} \label{def-fixed-point}
    A point process $\Pi$ is called a fixed point\footnote{Since this is a non-standard definition, we here and henceforth use the term \textit{fixed points} to avoid any confusion with invariant distributions.} of BBM with drift $\lambda$ if for all $f\in C_c^+(\mathbb{R})$ and for all $t\geq 0$,

    \begin{equation}\label{fixed-point-def}
        \mathbb{E}[e^{-\langle f, \theta_t^\lambda \rangle}] = \mathbb{E}[e^{-\langle f, \theta_0^\lambda \rangle}],
     \end{equation} 
    where $\theta_t^\lambda$ is given by \eqref{theta_t-def} with $\theta_0^\lambda = \Pi$. 
    
\end{definition}
This definition imposes no assumption on \(\Pi\) beyond almost-sure local
finiteness. It also implies that, when \(\theta_0^\lambda=\Pi\),
\(\theta_t^\lambda\) is almost surely locally finite for every \(t\geq0\).
Indeed, fix compactly supported continuous function \(g\in C_c^+(\mathbb R)\) and set
\(X_t=\langle g,\theta_t^\lambda\rangle\in[0,\infty]\). By
\eqref{fixed-point-def},
\[
  \mathbb E[e^{-\alpha X_t}]
  =\mathbb E[e^{-\alpha\langle g,\Pi\rangle}]
  \longrightarrow 1\qquad(\alpha\downarrow0).
\]
On the other hand, monotone convergence gives
\[
  \lim_{\alpha\downarrow0}\mathbb E[e^{-\alpha X_t}]
  =\mathbb P(X_t<\infty).
\]
Thus \(X_t<\infty\) almost surely. Choosing, for each compact \(K\), a function
\(g\in C_c^+(\mathbb R)\) satisfying \(g\geq1\) on \(K\), and then using a
countable compact exhaustion of \(\mathbb R\), proves that
\(\theta_t^\lambda\) is almost surely locally finite. In particular, the law of $\theta_t^\lambda$ seen as a point process is well defined and \eqref{fixed-point-def} implies that these laws are invariant in $t\geq 0$. As an intermediate step towards characterizing fixed points $\Pi$ we will in fact prove that fixed points $\Pi \in \mathcal{N}_2$ almost surely. Using \eqref{N_2-is-fixed}, this reconfirms that $\theta_t^\lambda$ started from $\Pi$ is indeed locally finite almost surely, and it also reassures that the space $\mathcal{N}_2$ is a natural state space for defining the Markov process $\theta_t^\lambda$. \\     

It can be easily shown that if $\lambda \in [0, \sqrt{2})$, then on the event $\theta \neq 0$, $\theta_t^\lambda([-K, K]) \overset{d}{\to} +\infty$ for all $K>0$ (one way of proving this is by using \eqref{Laplace-comp-using-F-KPP-Step-1} and applying Lemma \ref{Bramson-converges-to-1}). This implies that there are no non-trivial fixed points for BBM with drift $\lambda \in [0, \sqrt{2})$. The characterization of fixed points $\Pi$ for the supercritical drift $\lambda > \sqrt{2}$ was given in \cite{Kabluchko} under an additional assumption that $\Pi$ has locally finite intensity, i.e. $\mathbb{E}[\Pi(K)] < \infty$ for all compact sets $K \subset \mathbb{R}$.
\textcolor{black}{It is proved in Section \ref{Section Additional remarks}, Proposition \ref{Prop_intensity implies top particle}} that if $\Pi$ is a fixed point of BBM with drift $\lambda >\sqrt{2}$ with locally finite intensity, then it also satisfies $\mathbb{E}[\Pi([0,\infty))] < \infty$.
In particular, $\Pi([0,\infty)) <\infty$ almost surely, which is equivalent to $\Pi$ having a finite top particle almost surely. Hence, the assumption of locally finite intensity is a stronger and more restrictive assumption than the assumption of having a finite top particle. In the critical case $\lambda = \sqrt{2}$, the assumption of locally finite intensity for fixed points $\Pi$ is not appropriate since we a posteriori know that fixed points in this case must have locally infinite intensity. In \cite{Shekhar_FP} a characterization of fixed points for $\lambda = \sqrt{2}$ was given under the weaker and less restrictive assumption of finite top particle almost surely. In both critical (resp. supercritical) cases, the fixed points are given by random shifts of the so-called critical (resp. supercritical) extremal point process of BBM; see section \ref{Subsection Extremal Process} for its precise definition. 
In this article, we give a full characterization of fixed points $\Pi$ for both supercritical and critical cases $\lambda \geq \sqrt{2}$ without any additional assumption on $\Pi$ besides it being locally finite almost surely. The answer still turns out to be random shifts of critical/ supercritical extremal point processes. Hence, there are no exotic fixed points of BBM that violate the finite top particle condition. \\

\subsection{Main results} \label{main-results-subsection}


Let \(\mu=\lambda+\sqrt{\lambda^2-2}\) for
\(\lambda\geq\sqrt2\). Kabluchko \cite{Kabluchko} proved that, when
\(\lambda>\sqrt2\), the BBM evolution started from
\[
 \operatorname{PPP}\!\left(
   \frac1{\sqrt{2\pi}}e^{-\mu x}\, dx\right)
\]
converges in distribution, in the vague topology on point processes, to a
nontrivial process denoted by
\(\widetilde{\mathcal E}_\infty^\lambda\). In the critical case
\(\lambda=\sqrt2\), \cite{ABK13} proved the analogous convergence from
\[
 \operatorname{PPP}\!\left(
 \sqrt{\frac2\pi}(-x)e^{-\sqrt2x}\mathbf1_{\{x<0\}}\,dx\right)
\]
to a nontrivial process \(\widetilde{\mathcal E}_\infty^{\sqrt2}\).
We call these limits the supercritical and critical extremal point processes,
respectively, and abbreviate
\(\widetilde{\mathcal E}_\infty^{\sqrt2}\) as
\(\widetilde{\mathcal E}_\infty\).
See section \ref{Subsection Extremal Process} for a more detailed description. 

Our main result is as follows. For a point process $\theta = \sum_{i \in I} \delta_{x_i}$, we write $\theta(\cdot - S)$ for the point process obtained by shifting atoms $x_i$ by $S$, i.e. $\theta(\cdot - S) = \sum_{i \in I} \delta_{x_i + S}$\footnote{If the minus sign in front of $S$ on the left hand side causes confusion to a reader, note that in the notation $\theta(\cdot)$, we treat $\theta$ as a measure which is evaluated on Borel subsets of $\mathbb{R}$.}.

\begin{theorem}\label{main-thm}
Let $\lambda \geq \sqrt{2}$ and $\Pi$ be a point process. Then $\Pi$ is a fixed point of BBM with drift $\lambda$ in the sense of Definition \ref{def-fixed-point} if and only if there exists a $[-\infty, \infty)$-valued random variable $S$ independent of $\widetilde{\mathcal E}_\infty^\lambda$ such that 
\[ \Pi \overset{d}{=} \widetilde{\mathcal E}_\infty^\lambda (\cdot - S). \]
\end{theorem}

\begin{remark} \label{infinite-S}

Allowing \(S\) to take the value \(-\infty\) may at first appear unusual. On
\(\{S=-\infty\}\), we interpret
\(\widetilde{\mathcal E}_\infty^\lambda(\cdot-S)\) as the empty point process:
the base extremal process has an almost surely finite rightmost particle, so
all of its atoms are shifted to \(-\infty\). This convention is consistent
with the possibility that a fixed point is empty with positive probability.
The empty point process is itself a fixed point.
Also, note that a fixed point, conditional on being non-empty, remains a fixed point. One can hence, without loss of generality, assume that the fixed point $\Pi$ is non-empty, i.e. $\Pi(\mathbb{R}) \neq 0$ almost surely. Under this condition, the corresponding random variable $S$ will be a proper real-valued random variable. 

\end{remark}

\vspace{2mm}

Theorem \ref{main-thm} will be derived as a consequence of addressing another related problem of the \textit{domain of attraction of fixed points of BBM}. For a point process $\Pi$,
we define its domain of attraction as follows: 

\begin{definition}\label{Def_DOA}

    For a point process $\Pi$, we say a point process $\theta$ is in the domain of attraction of $\Pi$, written as  $ \theta \in \textrm{Dom}_{\lambda}(\Pi)$, if $\theta_t^\lambda$ given by \eqref{theta_t-def} started from $\theta_0^\lambda = \theta$ converges to $\Pi$ in the sense that for all $f \in C_c^+(\mathbb R),$ 
    \begin{equation}
        \mathbb E[e^{-\langle f , \theta_t^\lambda \rangle}]\longrightarrow \mathbb E[e^{-\langle f , \Pi \rangle}]
    \end{equation}
    as $t \to \infty.$
\end{definition}

Note that the above definition, similarly to Definition \ref{def-fixed-point}, bypasses the issue that $\theta_t^\lambda$ can possibly be locally infinite. Nevertheless, as it turns out, if a point process $\theta \in Dom_{\lambda}(\Pi)$ for a point process $\Pi$, then $\theta_t^\lambda$ must be locally finite almost surely. It is easy to heuristically support this claim: since the initial point process $\theta$ and the limiting point process $\Pi$ are locally finite, the intermediate collection of points $\theta_t^\lambda$ must also be locally finite. We will, in fact, prove that for $\theta \in Dom_{\lambda}(\Pi)$, $\theta \in \mathcal{N}_2$ almost surely, which, using \eqref{N_2-is-fixed}, implies that $\theta_t^\lambda$ is a well-defined point process. Hence, Definition \ref{Def_DOA} implies that $\theta_t^\lambda$ converges in distribution to $\Pi$ as point processes. \\

It follows easily from the Markovian nature of the evolution $\theta \mapsto \theta_t^\lambda$ that for a point process $\Pi$, $Dom_{\lambda}(\Pi) \neq \emptyset$ if and only if  $\Pi$ is a fixed point of BBM with drift $\lambda$ in the sense of Definition \ref{def-fixed-point}. Hence, without loss of generality, we assume in Definition \ref{Def_DOA} that $\Pi$ is a fixed point of BBM with drift $\lambda$. The proof of Theorem \ref{main-thm} will run in parallel and will follow as a result of addressing the following problem of characterizing $Dom_{\lambda}(\Pi)$ of fixed points $\Pi$. Let us first coin the following definition. The significance of this definition will be addressed in the next section.

\begin{definition}\label{def-lacunary}
\begin{enumerate}
    \item A point process $\theta$ is called \textbf{lacunary} on $(0,\infty)$ if there exist constants $c,d \in (0, \infty)$, $c<d$, such that  
    \begin{equation} \label{lacunary-1}
        \theta([cx, d x ]) \overset{\mathbb{P}}{\rightarrow} 0
    \end{equation} 
    as $x \rightarrow \infty $\footnote{The intuitive meaning of this definition is that atoms of the point process $\theta$ in $(0,\infty)$ are extremely spaced out and have large gaps, hence the choice of the word \textbf{lacunary}.}. 

    \item A point process $\theta$ is called \textbf{exponentially-lacunary} on $(0,\infty)$ if there exist constants $c,d \in (0, \infty)$, $c<d$, such that 
    \begin{equation}\label{lacunary-2}
         \theta([c \log x, d x ]) \overset{\mathbb{P}}{\rightarrow} 0
    \end{equation}
    as $x \rightarrow \infty $. 
\end{enumerate}
\end{definition}

Let us list the conditions appearing in the next result. \textcolor{black}{Throughout this paper, we will refer to a family of random variables $\{X_t\}_{t>0}$ as \textit{tight} if $\sup_{t\geq 1} \mathbb{P}(|X_t|\geq K ) \to 0$ as $K\to +\infty$. In particular, we do not put any tightness assumption on $X_t$ for $t \in (0,1)$.} For $\lambda >\sqrt{2}$, we assume  \\

\begin{enumerate}
    \item [(H1):] $\theta$ is \textbf{lacunary} on $(0,\infty)$ and \textcolor{black}{$\theta \in \mathcal N_2$} \footnote{\textcolor{black}{Note that, $\mathcal N_2$ is defined in \eqref{N_2-def}.}} a.s.
    \item [(H2):] For all $K>0$,  
    \[\int_{-\infty}^{\infty} \mathbb{P}(|x+ B_t - \lambda t | \leq K)\theta(dx) \textrm{ and }\int_{- \infty}^0 \frac{1}{\sqrt{t}}e^{(1-\frac{\lambda^2}{2})t}e^{\lambda x}e^{-\frac{x^2}{2t}}\theta(dx)\] 
    are tight in $t\geq 1$.

 \item [(H3):] As $t \to \infty$ and then $L\to \infty$, 
 \[ Z_{t,L}:=\int_{(\lambda -\mu) t - L\sqrt{t}}^{(\lambda -\mu)t + L\sqrt{t}} \frac{1}{\sqrt{t}}e^{(1-\frac{\lambda^2}{2})t}e^{\lambda x}e^{-\frac{x^2} {2t}}\theta(dx) \overset{d}{\to} Z,\]
 where $\mu = \lambda + \sqrt{\lambda^2-2}$ and $Z$ is some non-negative random variable. 
   
\end{enumerate}

\vspace{6mm}
For $\lambda = \sqrt{2}$, we assume  \\

\begin{enumerate}
    \item [(A1):] $\theta$ is \textbf{exponentially-lacunary} on $(0,\infty)$ and \textcolor{black}{$\theta \in \mathcal N_2$ a.s.}
    \item [(A2):] For all $K>0$,  
    \[\int_{-\infty}^{\infty} \mathbb{P}(|x+ B_t - \sqrt{2} t | \leq K)\theta(dx) \textrm{ and } \int_{-\infty}^{\frac{1}{4} \log t }\frac{1}{t^{3/2}}(-x + \log t)e^{\sqrt{2}x} e^{-\frac{x^2}{2t}} \theta(dx)\] 
    are tight in $t\geq 1$. 
    \item [(A3):] As $t\to \infty$, 
    \[ \int_{-\infty}^{0}\frac{1}{t^{3/2}}(-x)e^{\sqrt{2}x} e^{-\frac{x^2}{2t}} \theta(dx) \overset{d}{\to} Z,\]
    \textcolor{black}{where} $Z$ is some non-negative random variable. 
\end{enumerate}

\begin{theorem} \label{main-thm-2}
    Let \(\lambda\geq\sqrt2\), let \(\Pi\) be a fixed point of BBM with drift parameter \(\lambda\) (that is, particle drift \(-\lambda\)), and let \(\theta\) be a point process.
\begin{enumerate}
\item Suppose \(\lambda>\sqrt2\), and write
\[
 \Pi\overset d=
 \widetilde{\mathcal E}_\infty^\lambda(\cdot-S_\Pi),
 \qquad Z_\Pi:=e^{\mu S_\Pi},
 \qquad \mu=\lambda+\sqrt{\lambda^2-2},
\]
where \(S_\Pi\in[-\infty,\infty)\) is independent of
\(\widetilde{\mathcal E}_\infty^\lambda\). Then
\(\theta\in\operatorname{Dom}_\lambda(\Pi)\) if and only if \(\theta\)
satisfies (H1) and (H2), and (H3) holds with a limit \(Z\) such that
\[
 Z\overset d=Z_\Pi.
\] 

\item Suppose \(\lambda=\sqrt2\), and write
\[
 \Pi\overset d=\widetilde{\mathcal E}_\infty(\cdot-S_\Pi),
 \qquad Z_\Pi:=e^{\sqrt2S_\Pi}.
\]
Then \(\theta\in\operatorname{Dom}_{\sqrt2}(\Pi)\) if and only if
\(\theta\) satisfies (A1) and (A2), and (A3) holds with
\(Z\overset d=Z_\Pi\).
\end{enumerate}

\end{theorem}

\begin{remark}\label{about-lacunary}

It can be easily seen that if \eqref{lacunary-1} or \eqref{lacunary-2} holds for some $c,d \in (0,\infty)$, then it holds for all $c,d \in (0,\infty)$. Also, \eqref{lacunary-1} is equivalent to $\theta([x, \alpha x]) \overset{\mathbb{P}}{\to} 0$ as $x\to \infty$ for some $\alpha>1$ (and hence for all $\alpha >1$). As such, the condition \eqref{lacunary-1} \textcolor{black}{is closely related to} slowly varying functions, see \cite{Bingham_Regular_variation}. Indeed, if $F(x) = e^{\theta([0,x])}$, then \eqref{lacunary-1} implies $F(\alpha x )/F(x) \overset{\mathbb{P}}{\rightarrow} 1$ as $x\to \infty$ for some $\alpha >1$.
\end{remark}

\begin{remark}\label{comparison-to-top-particle-assumption}
It is instructive to compare condition \eqref{lacunary-1} with the finite top particle condition $\theta([0,\infty)) < \infty$ almost surely. In combination with the fact that $\theta([0,x])$ is integer valued, \eqref{lacunary-1} is a rather stringent condition. In fact, if $\theta$ is a deterministic locally finite integer-valued measure, \eqref{lacunary-1} holds if and only if $\theta([cx, dx]) = 0$ for all $x$ large enough. This holds if and only if $\theta ([0,\infty)) < \infty$, i.e., $\theta$ has a finite top particle. However, if we ignore the integer valued constraint on $\theta$ and view it as just a positive measure, by utilizing the above mentioned connection to slowly varying functions, one can easily come up with examples of positive measures satisfying \eqref{lacunary-1} and $\theta([0,\infty)) =\infty$, e.g. $ \theta (dx) = \frac{1}{x\log x}1_{x \geq 2}dx$. Combining these observations, one can come up with examples of truly random point processes satisfying \eqref{lacunary-1} but not having a finite top particle, e.g. $\theta = PPP( \frac{1}{x\log x}1_{x \geq 2 }dx)$. Constructing an example of a point process satisfying \eqref{lacunary-2} but not having a finite top particle is only slightly more intricate. Let $x_n$ be the sequence defined by $x_1 =e$ and $x_{n+1} = e^{x_n}$ (it is the \textit{power tower} sequence). Let $G(x) = \sum_{n=1}^{\infty} 1_{x_n \leq x}$ and $m$ be the measure on $(0,\infty)$ defined by $m(0,x] = \log G(x)$. Note that \eqref{lacunary-2} is equivalent to $\theta([x,e^x]) \overset{\mathbb{P}}{\to} 0$ as $x\to \infty$. It can be easily checked that $m$ is an infinite measure satisfying $m([x,e^x]) \to 0$ as $x\to \infty$, which implies that $\theta = PPP(m(dx))$ satisfies \eqref{lacunary-2} and $\theta([0,\infty))= +\infty$ almost surely.   

\end{remark}

\begin{remark}

As it can be seen from Theorem \ref{main-thm-2}, it is not necessary that $\theta \in Dom_{\lambda}(\Pi)$ has a finite top particle. However, if $\theta \in Dom_{\lambda}(\Pi)$ is deterministic, it then follows from Remark \ref{comparison-to-top-particle-assumption} that $\theta$ must have a top particle. Hence, for deterministic $\theta \in Dom_{\lambda}(\Pi)$, conditions (A1),(A2) are redundant for part-(2) of Theorem \ref{main-thm-2} and they can be replaced with the top particle assumption $\theta([0,\infty)) <\infty$. As such, it can be compared to the results of \cite{Shekhar_Domain}.

\end{remark}

\subsection{Heuristic ideas and outline of the proof.}

Recall from Section \ref{main-results-subsection} the definition $\widetilde{\mathcal E}_\infty^\lambda$ as the limit of $\theta_t^\lambda$ starting from 
$\theta = PPP(\frac{1}{\sqrt{2\pi}}e^{-\mu x } dx)$ for $\lambda > \sqrt{2}$ and starting from $\theta = PPP(\sqrt{\frac{2}{\pi}}(-x)e^{-\sqrt{2}x}1_{x < 0} dx)$ for $\lambda=\sqrt 2$. Since we expect the fixed points $\Pi$ to be random shifts of $\widetilde{\mathcal E}_\infty^\lambda$, it is natural to expect that for $\theta \in Dom_{\lambda}(\Pi)$, the number of particles of $\theta$ at a bounded distance from $x$ should be $\approx e^{-\mu x }$ for $\lambda >\sqrt{2}$ and $\approx (-x)e^{-\sqrt{2}x}$ for $\lambda = \sqrt{2}$ as $x\rightarrow -\infty$. A rigorous formulation of this heuristic was developed for the critical case
in \cite{Shekhar_FP} and \cite{Shekhar_Domain}; analogous arguments apply
when \(\lambda>\sqrt2\). 
This is reflected in the conditions (H3)/(A3) appearing in Theorem \ref{main-thm-2}. 
However, for a $\theta \in Dom_{\lambda}(\Pi)$, the structural information about particles of $\theta$ at a bounded distance from $x$ as $x \to +\infty$ was not understood in \cite{Shekhar_FP},\cite{Shekhar_Domain}. 
The results of \cite{Shekhar_FP}, \cite{Shekhar_Domain} were hence proven under the additional assumption that $\theta$ has a finite top particle almost surely. The finite-rightmost-particle assumption bypassed this unresolved issue. The main new contribution of the present paper is to remove that assumption.
This is given by the lacunary conditions appearing in Theorem \ref{main-thm-2} as per Definition \ref{def-lacunary}. \\      

We now elaborate on the significance of lacunary conditions appearing in Theorem \ref{main-thm-2}. For a fixed point $\Pi$ and $\theta \in Dom_{\lambda}(\Pi)$, it is natural to look at the cloud of \textcolor{black}{leading} particles of $\theta_t^\lambda$. More precisely, let $M(\theta_t^\lambda)$ be the thinned collection of points given by 
\begin{equation}\label{def-thinning-M}
    M(\theta_t^\lambda):= \sum\limits_{x_i \in \theta} \delta_{x_i+M_t^i-\lambda t},
\end{equation}
where $M_t^i$ is the maximum particle of BBM $\{\chi_k^i(t)\}_{k \leq n^i(t)}$ appearing in the definition \eqref{theta_t-def} of $\theta_t^\lambda$. While studying the large time behaviour of $\theta_t^\lambda$, it is natural to first study the large time behaviour of $M(\theta_t^\lambda)$. Note that $M(\theta_t^\lambda)$ is a (non-Markovian) independent particle system since $\{M_t^i\}$ are i.i.d. over $i \in I$. For $\theta \in Dom_{\lambda}(\Pi)$, since $\theta_t^\lambda$ converges to the non-trivial limit $\Pi$, it is natural to anticipate that $M(\theta_t^\lambda)$ converges to a non-trivial limit, and since it is an independent particle system, the limit must be a Poisson point process. However, it is a priori not clear what would be the intensity measure of this limiting Poisson point process. \\

Instead of picking the maximum particle $M_t^i$, it would also be beneficial to pick a uniformly chosen particle from the collection $\{\chi_k^i(t)\}_{k \leq n^i(t)}$. One can choose such a particle adaptively as follows. We attach a tag to the first particle of the BBM $\{\chi_k^i(t)\}_{k \leq n^i(t)}$ that starts from the origin. At each branching event, the tag is transferred to one of the children uniformly at random. Let $B_t^i$ denote the location of the tagged particle at time $t$. Clearly, $B_t^i$ is a standard Brownian motion. Instead of choosing $M_t^i$, we can also choose $B_t^i$ to obtain a thinned collection of points 
\begin{equation}\label{def-thinning-U}
    U(\theta_t^\lambda) := \sum\limits_{x_i \in \theta} \delta_{x_i+B_t^i-\lambda t}
\end{equation}

The advantage of picking a uniformly chosen particle is that the evolution $\theta \mapsto U(\theta_t^\lambda)$ is Markovian. Such systems were introduced and studied by Liggett \cite{Liggett}. By the results of Liggett \cite{Liggett}, the limit of $U(\theta_t^\lambda)$ as $t \to \infty$ is a Poisson point process, which is also an invariant distribution for the evolution $\theta \mapsto U(\theta_t^\lambda)$. Hence, the intensity measure of this limiting Poisson point process satisfies a Choquet-Deny equation, see Section 1.5.1 of \cite{Shekhar_FP} for details. Solving this Choquet-Deny equation yields that the intensity measure has to be a linear combination $Z_\infty e^{-2\lambda x} dx + Y_\infty dx$ for some (possibly random) non-negative constants $Z_\infty, Y_\infty$, see \cite{Shekhar_Liggett} for an alternative direct approach. Note that on the event $Y_\infty >0$, the $PPP(Z_\infty e^{-2\lambda x} dx + Y_\infty dx)$ does not have a finite top particle. \\

For the thinning $M(\theta_t^\lambda)$, since it is non-Markovian, the limiting Poisson point process obtained as the limit of $M(\theta_t^\lambda)$ as $t\to \infty$ is not necessarily an invariant distribution for the evolution $\theta \mapsto M(\theta_t^\lambda)$. Hence, we do not have a Choquet-Deny type equation at our disposal to determine the intensity measure of this limiting PPP. However, resembling the case of $U(\theta_t^\lambda)$, we guess and conjecture that the intensity measure in this case has to be of the form $A_\infty e^{-\mu x} dx + B_\infty dx$, where $\mu = \lambda + \sqrt{\lambda^2 -2}$. We emphasize that $B_\infty$ can be positive with positive probability. For example, if $\theta_0^\lambda =\theta = PPP(dx)$, then $M(\theta_t^\lambda)$ is also distributed as $PPP(dx)$ for all $t>0$ ($PPP(dx)$ is invariant for any independent particle system, Markovian or non-Markovian). We again note that on the event $B_\infty >0$, $PPP(A_\infty e^{-\mu x} dx + B_\infty dx)$ does not have a finite top particle. \\

We infer from the discussion above that for a fixed point $\Pi$ and $\theta \in Dom_{\lambda}(\Pi)$, it may be a priori possible that $M(\theta_t^\lambda)$ converges to a point process which does not have a finite top particle. Nevertheless, we claim that the fixed point $\Pi$ must have a finite top particle almost surely. The heuristic but intuitive reason behind this is as follows. To obtain back $\theta_t^\lambda$ from $M(\theta_t^\lambda)$, one has to recollect the cloud of particles within a bounded distance from $M_t^i$ for each $i\in I$. Let us call these particles \textit{decoration particles} since, informally speaking, in the limit $t\to \infty$, conditional on $M_t^i$ being large, these particles form the decoration point process in the extremal point process, see Section \ref{Section_Prelim} for a more precise statement. The number of decoration particles grows exponentially with time $t >0$. Hence, unless $B_\infty$ appearing in the limit of $M(\theta_t^\lambda)$ is zero, the exponential contributions from these particles would force the limit of $\theta_t^\lambda$, which is $\Pi$, to be locally infinite, which is a contradiction. Also, to counterbalance the exponential contribution from the decoration particles, the atoms of the initial point process $\theta_0^\lambda = \theta$ must be very sparse in $(0,\infty)$. This is quantified using the \textbf{lacunary} point processes as in Definition \ref{def-lacunary}. \\

The proof strategy is similar to the one used in \cite{Shekhar_FP} along with some new ideas. We make an educated guess and divide $\mathbb{R}$ into a certain number of pieces, see \eqref{Def-sets-E},\eqref{Def-sets-F}, and track the contribution of atoms of $\theta$ in each piece to $\theta_t^\lambda$ for large time $t \gg 1$. For $\lambda >\sqrt{2}$, we show that the only non-trivial contribution to $\theta_t^\lambda$ comes from atoms of $\theta$ in an $O(\sqrt{t})$ neighbourhood of $-\sqrt{\lambda^2 -2}t$. Furthermore, we show that the contributions from the rest of the atoms of $\theta$ converge to zero in the limit $t\to \infty$. For the contribution coming from the region $\{x >bt\}$, we use the Brownian thinning $U(\theta_t^\lambda)$ given by \eqref{def-thinning-U} and prove a tightness estimate on $\theta$, see Proposition \ref{crucial-tightness-prop}.
This is an essential piece of new argument since in the region $\{x>bt\}$ we could not rely on estimates for F-KPP solutions (note that Bramson's $\psi$ estimate given by Proposition \ref{Bramson-psi-estimate-for-one-sided-f} does not hold in a neighbourhood of $-\infty$). We overcome this difficulty by using the Brownian thinning
\(U(\theta_t^\lambda)\). 
The contribution coming from region $[at, bt]$ for $a,b>0$ is handled by proving the lacunary property of $\theta$ as explained in the previous paragraph, see Proposition \ref{lacunary-prop}. The remaining contributions are handled by establishing tightness estimates on $\theta$ obtained via the thinning $M(\theta_t^\lambda)$. We can here rely on the precise estimates on solutions to the F-KPP equation given by Bramson's $\psi$ estimate; see section \ref{Section_Prelim}. The argument using these tightness estimates to prove that the remaining contributions vanish in the limit $t\to \infty$ is similar to that in \cite{Shekhar_FP}: a non-zero contribution at time $t$ would force an unusually large contribution at a larger or smaller time $u$, which would contradict the tightness. The proof for $\lambda = \sqrt{2}$ follows using similar arguments.

\subsection{Organization of the paper.}

The rest of the paper is organized as follows: In Section \ref{Section_Prelim}, we recall some well-known results on BBM as well as some estimates used in later sections.
In Section \ref{Section Properties theta in Don}, we list some properties of a point process when it belongs to the domain of attraction of a fixed point.
We prove Theorem \ref{main-thm} and \ref{main-thm-2} in Section \ref{Section proof supercritical} and \ref{Section proof critical} for $\lambda > \sqrt 2$ and $\lambda=\sqrt 2$ respectively.
We include some relevant discussions on supercritical fixed points and Tauberian conditions in Section \ref{Section Additional remarks}.
We conclude this paper with Appendix \ref{Section Appendix}, where we identify the explicit structure of supercritical fixed points.

\subsection{Notations.}
We write $f \lesssim g$, to mean $f \le C g$ for a fixed constant $C.$ As
\(t\to\infty\), the notation \(a(t)=o(b(t))\) means
\(a(t)/b(t)\to0\), whereas \(a(t)=O(b(t))\) means that
\(|a(t)|\leq C|b(t)|\) for all sufficiently large \(t\), for some
\(C<\infty\).


\subsection{Acknowledgments.}
The authors would like to thank Zenghu Li for helpful suggestions. XC and SZ are supported by supported by National Key R\&D Program of China
(No. 2022YFA1006500). CA and SA are supported through grant project no. RTI4014 of the Department of Atomic Energy, Government of India.

\section{Preliminaries} \label{Section_Prelim}

\textcolor{black}{In this section, we introduce the state space of point processes that we work with and review some classical results in probability theory. We then present some basic facts on BBM and the associated F-KPP equation, and collect several estimates for F-KPP solutions. Finally, we introduce the extremal point processes of BBM. The results in this section are mainly extracted from \cite{Bramson78, Bramson83, ABK13, ABBS13, Bovier_book, Shekhar_Domain} and will be used in the proof of the main theorems.}

\subsection{State space and the Laplace functional} \label{Section Laplace}

Let $\mathcal{N}$ be the space of all locally finite point measures (i.e., they are integer-valued) on $\mathbb{R}$,  
 \[\mathcal{N} = \bigl\{ \eta = \sum_{i \in I} \delta_{x_i} \hspace{2mm}|\hspace{2mm} \eta(K) <\infty \hspace{2mm} \textrm{for all compact sets} \hspace{2mm} K\subset \mathbb{R}\bigr\}.\]
$\mathcal{N}$ is equipped with the topology of vague convergence: for $\eta_n, \eta \in \mathcal N$, we say $\eta_n \stackrel{v}{\rightarrow} \eta$ if for all $f \in C_c^+(\mathbb R)$, $\langle f,\eta_n\rangle \to \langle f,\eta\rangle$ as $n\to\infty$.

A point process $\theta$ is a $\mathcal{N}$-valued random variable. It is well known that the law of $\theta$ is characterized by its Laplace functional, i.e., $\theta \stackrel{d}{=} \theta_0$ if and only if $\psi_\theta(f)=\psi_{\theta_0}(f)$ for all $f \in C_c^+(\mathbb R)$, where $\psi_\theta(f)$ is the Laplace functional of $\theta$ defined by 
\begin{align*}
\psi_\theta(f):=\mathbb E[e^{-\langle f, \theta \rangle}].
\end{align*}

In order to utilize the connection between the BBM and F-KPP equations, we will often be working with functions of the following form:
\begin{equation} \label{Form of f}
f(x)= \sum_{k=1}^n a_k 1_{x \ge b_k}, a_k>0, b_k \in \mathbb R.
\end{equation}
\textcolor{black}{We will use $f$ of the form \eqref{Form of f} instead of $f \in C_c^+(\mathbb R)$ as it is sufficient to consider this class of test functions for weak convergence of point processes having finite right tails (that is,
\(\theta([b,\infty))<\infty\) for every \(b\in\mathbb R\)); see the remark on page 129 in \cite{Bovier_book}.}

Consider the subspace $\mathcal{N}_2 \subset \mathcal{N}$ defined by 
\begin{equation}\label{N_2-def}
     \mathcal{N}_2 := \biggl\{ \eta \in \mathcal{N}\hspace{2mm} \biggl|\hspace{2mm} \int_{-\infty}^{\infty} e^{-\alpha x^2} \eta(dx) <\infty \hspace{2mm} \textrm{for all} \hspace{2mm} \alpha >0\biggr\}.
 \end{equation}

The space $\mathcal N_2$ is invariant under the BBM:

\begin{lemma} [Easy adaptation of Lemma $3.7$-\cite{Shekhar_Domain}]\label{Proposition M_2}
If $\theta_0^\lambda = \theta \in \mathcal N_2$ a.s., then for all $t>0$, $\theta_t^\lambda \in \mathcal N_2$ a.s. In particular, for such $\theta_0^\lambda = \theta$, $\theta_t^\lambda$ is a well-defined point process. Furthermore, if $\theta \in \mathcal{N}_2$ and $\theta([0,\infty)) <\infty$ a.s., then $\theta_t^\lambda([0,\infty)) <\infty$ a.s. for $t>0$. 
\end{lemma}

\subsection{Some classical results in probability theory.}

The following lemma is classical, see [Theorem $2$-Chapter$13$-\cite{Feller-book-Vol2}]. However, notice the nuance that $X_t$ is allowed to take $+\infty$ value.  

\begin{lemma}\label{Feller's lemma}
    Let $X_t,X$ be a family of $[0,\infty]$-valued random variables such that $\mathbb{P}\left(X <\infty\right) = 1$. Assume that for all $\rho>0$,
    \begin{equation}
        \mathbb{E}[e^{-\rho X_t}] \longrightarrow \mathbb{E}[e^{-\rho X}] \quad \text{as} \quad t\to\infty.
    \end{equation}
    Then, 
    \begin{equation}
        \mathbb{P}\left(X_{t} = +\infty\right) \to 0 \quad \text{as} \quad t\to\infty,
    \end{equation} 
    and conditioned on the event $\{X_{t}<\infty \}$, $X_{t}$ converges weakly to $X$. In particular, for all $x \in [0,\infty)$ which is a point of continuity for  $\mathbb{P}(X\leq x)$, $\mathbb{P}\left(X_{t}\le x\right) \to \mathbb{P}\left(X\le x\right)$ as $t\to\infty$.
\end{lemma}

\begin{lemma} [Lemma 3.2 of \cite{Shekhar_FP}]\label{Lemma_bernoulli}
Let $I$ be a countable set and $\{X_i\}_{ i\in I }$ be independent Bernoulli random variables such that $\mathbb E[X_i]=p_i$. Let $X_I=\sum_{i \in I}X_i.$ 
\begin{enumerate}[label=(\alph*)]
    \item If $\mathbb E[X_I]=\sum_{i \in I}p_i$ is finite, then for any $\epsilon \in (0,1)$,
\begin{align*}
\mathbb P(|X_I-\mathbb E[X_I]|\ge \epsilon \mathbb E[X_I])\le \frac{1}{\epsilon^2 \mathbb E[X_I]}.
\end{align*}

\item If $I$ is infinite and $\mathbb E[X_I]=\sum_{i \in I} p_i= \infty$, then  $X_I=\infty$ almost surely.
\end{enumerate}

\end{lemma}   

\subsection{BBM and F-KPP equation}\label{Section BBM F-KPP} For the binary BBM, recall that the positions of the particles alive at time $t$ are denoted by $\{\chi_k(t), 1 \le k \le n(t)\}$, where $n(t)$ is the number of particles alive at time $t$. The following lemma captures the connection between the BBM and the F-KPP equation. This lemma is due to \cite{McKean} but also appeared in \cite{Skorohod_Branching,Ikeda_1,Ikeda_2,Ikeda_3}.

\begin{lemma}[Lemma $5.5$ of \cite{Bovier_book}]\label{BBM-F-KPP-Connection}
For a measurable function $\phi:\mathbb{R} \to [0,1]$, let
\begin{equation}\label{u-def}
u_\phi(t,x) = 1 - \mathbb{E}[\prod_{k=1}^{n(t)}\{1-\phi(x- \chi_k(t))\}].
\end{equation}
Then $u = u_\phi$ solves the following F-KPP equation
\begin{equation}\label{F-KPP}
\partial_t u = \frac{1}{2}\partial_x^2 u + u- u^2
\end{equation}
 with the initial condition $u(0,x) = \phi(x)$.
 \end{lemma}
 
If $\phi (x) =1_{(-\infty, 0]}(x)$, we write the corresponding solution $u_\phi$ as $u_M$. The relation \eqref{u-def} gives, in particular, the distribution of the maximum particle $M_t$ of BBM: 
 \begin{equation}\label{F-KPP-Formula-For-M}
    u_M(t,x) = \mathbb{P}(M_t \ge x). 
\end{equation}
 
More generally, for $f:\mathbb{R} \to [0,\infty)$, it follows easily from \eqref{u-def} that 

\begin{equation}\label{Laplace-comp-using-F-KPP-Step-1}
     1- u_\phi(t, \lambda t -x) = \mathbb{E}[e^{-\sum_{k=1}^{n(t)} f(x + \chi_k(t) -\lambda t )}],
\end{equation}
where $\phi$ and $f$ are related by $\phi(x) = 1 - e^{-f(-x)}$. This observation leads to the following expression of the Laplace functional of BBM, and as the proof is quite straightforward, it is \textcolor{black}{omitted}.
\begin{lemma}\label{Formula-for-Laplace-of-theta_t}
    Let $\theta$ be a point process and  $\theta_t^\lambda$ as in \eqref{theta_t-def}. Then, for all $f\in C_c^+(\mathbb R)$ and $\phi(x):=1-e^{-f(-x)},$
    \begin{equation}\label{Laplace-comp-using-F-KPP}
\mathbb{E}[e^{-\langle f,\theta_t^\lambda\rangle}] = \mathbb{E}\biggl[\exp\biggl\{\int_{-\infty}^\infty\log(1- u_\phi(t, \lambda t -x))\theta(dx)\biggr\}\biggr].
\end{equation}
\end{lemma}

Besides the connection of BBM to F-KPP, we will also use the following \textit{many-to-one} lemma, see \cite{Harris_manytofew} for details. 

\begin{lemma}[The many-to-one lemma]\label{Lemma many to one}
Let $F: \mathbb R \to \mathbb R$ be a bounded measurable function. Then, 
    \begin{equation}
        \mathbb E[\sum_{k \le n(t)} F(\chi_k(t))]=e^t \mathbb E[F(B_t)],
    \end{equation}
    where $B$ is a standard one dimensional Brownian motion.
\end{lemma}

\subsection{Some estimates on solution to F-KPP equation.}

The following result is taken from \cite{Bramson83}.

\begin{lemma}[Proposition 3.4 of \cite{Bramson83}]\label{Bramson-converges-to-1}\footnote{This claim can also be verified in the modern language of extremal point processes $\widetilde{\mathcal{E}}_\infty^\lambda$. Notably, the underlying reason for \eqref{other-main-input-from-Bramson} to hold is that the number of particles of $\widetilde{\mathcal{E}}_\infty^\lambda$ at a bounded distance from $x$ blows up as $x\to -\infty$: it grows like $ e^{-\mu x}$ for $\lambda >\sqrt{2}$ and $\approx (-x)e^{-\sqrt{2}x}$ for $\lambda = \sqrt{2}$.}
    Let $\phi:\mathbb{R}\to [0,1]$ be a measurable function such that $\phi(x) \geq \eta$ for all $x\in (a,b)$ \textcolor{black}{with} some $a,b \in \mathbb{R}$ and $\eta >0$. Let $u_\phi(t,x)$ be the solution to F-KPP equation with initial condition $\phi$. Then, for any $\delta >0$, as $t \to \infty$,
\begin{equation}\label{other-main-input-from-Bramson}
u_\phi\left(t, x\right) \to 1
\end{equation}
uniformly in $x$ such that $|x| \le \sqrt{2}t - \bigl(\frac{3}{2\sqrt{2}} + \delta\bigr)\log t$.
\end{lemma}

The above lemma gives the region of $x$ where $u_\phi$ converges to $1$. On the other hand, it is well known (see e.g. Lemma \ref{when-does-u-go-to-zero}) that for $f$ either compactly supported or of the type \eqref{Form of f}, $u_\phi(t, \sqrt{2}t + x) \to 0$ uniformly over $x\geq - c \log (t)$ as $t \to \infty$, where $c$ is a small enough constant. The following Bramson's $\psi$ function is a powerful tool that allows us to obtain sharp estimates on the solutions of the F-KPP equation in the region where $u_\phi$ converges to zero. 

\begin{proposition}[Proposition 8.3 of \cite{Bramson83}, Proposition 4.3 of \cite{ABK13}]\label{Bramson-psi-estimate-for-one-sided-f}
For a function $f$ of type \eqref{Form of f} and $\phi(x) = 1-e^{-f(-x)}$, let $u_{\phi}$ be the solution to the F-KPP equation with initial condition $\phi$. Define for $z\in \mathbb R$ and $t>r>0$,
\begin{equation}
\psi(r,t,z+\sqrt 2t):=\frac{e^{-\sqrt 2z}}{\sqrt{2\pi(t-r)}}\int_0^\infty u_{\phi}(r,y+\sqrt 2r)e^{\sqrt 2y}e^{-\frac{(y-z)^2}{2(t-r)}}\{1-e^{-2y\frac{z+\frac{3}{2\sqrt 2}\log t}{t-r}}\}dy.
\end{equation}
Then, for all $r$ large enough (depending only on the initial condition $\phi$), $t\geq 8r$ and $z \geq 8r-\frac{3}{2\sqrt 2}\log t$,
\begin{equation}
\gamma_r^{-1} \psi(r,t,z+\sqrt 2t) \le u_{\phi}(t,z+\sqrt 2 t) \le \gamma_r \psi(r,t,z+\sqrt 2t),
\end{equation}
where $\gamma_r \downarrow 1$ as $r\to \infty$.
\end{proposition}

The above proposition has several important consequences: 

\begin{lemma}[Proposition 3.1 of \cite{bovier2014extremal}, Lemma 9.8 of \cite{Bovier_book}]\label{C(a) definition}
Let $f, \phi, u_\phi$ be as in Proposition \ref{Bramson-psi-estimate-for-one-sided-f}. Then, for $\alpha>0$ and $z= \alpha t+o(t)$,
\begin{equation}\label{C(a)}
C(f, \alpha) :=\lim_{t \to \infty} e^{\sqrt 2z} e^{\frac{z^2}{2t}}t^{\frac{1}{2}}u_{\phi}(t,z+\sqrt 2 t)
\end{equation}
exists, and is a strictly positive constant. The convergence in the right-hand side above is uniform for $\alpha$ in compact subsets of $(0,\infty)$. 
\end{lemma}

\begin{lemma}[Lemma 9.10 of \cite{Bovier_book}] \label{C(a)-def-2}
With $f, \phi, u_\phi,  C(f, \alpha)$ given as in Lemma \ref{C(a) definition}, it holds that, as $t\to \infty$\footnote{The range of integration for \eqref{second-def-C(a)} in [Lemma 9.10-\cite{Bovier_book}] is given to be $(0,\infty)$. However, since $u_{\phi}(t, \sqrt{2}t + x) \to 0$ as $t\to \infty$ pointwise in $x$ (see Lemma \ref{when-does-u-go-to-zero}), one can easily change the range of integration to $(-\infty, \infty)$ by invoking the dominated convergence theorem.},
\begin{equation}\label{second-def-C(a)}
\frac{1}{\sqrt{2\pi}} \int_{-\infty}^\infty u_\phi(t, \sqrt{2}t + z) e^{(\sqrt{2}+ \alpha)z - \alpha^2t /2} dz \to C(f, \alpha).     
\end{equation}
In particular, for $\lambda >\sqrt{2}$, $\mu =\lambda + \sqrt{\lambda^2-2}$, and $\alpha(\lambda) := \mu -\sqrt{2}$, by substituting $ z = (\lambda -\sqrt{2})t - x$, 
\begin{equation}\label{C(a)-specific-to-lambda}
    \frac{1}{\sqrt{2\pi}} \int_{-\infty}^{\infty} u_{\phi}(t, \lambda t -x)e^{-\mu x} dx \to C(f, \alpha(\lambda)). 
\end{equation}
Also, for any fixed $x \in \mathbb{R}$ (see equation (9.50) of \cite{Bovier_book}), 
\begin{equation}\label{shift-relation-for-C(a)}
    C(f(x + \cdot), \alpha(\lambda)) = e^{\mu x}C(f, \alpha(\lambda)).
\end{equation}

\end{lemma}

If $f$ is such that $\phi (x) = 1-e^{-f(-x)} =  1_{(-\infty, 0]}(x)$, motivated by the relation \eqref{F-KPP-Formula-For-M}, we write

\begin{equation} \label{rename-C(f)-for-M-for-supercritical}
    C(f, \alpha(\lambda)) = C_M(\alpha(\lambda)).
\end{equation}

The following quantitative upper and lower bounds hold as well.

\begin{lemma} \label{lemma_upper and lower for F-KPP}
\begin{enumerate}[label= (\alph*)]
    \item Let $f, \phi, u_\phi$ be as in Proposition \ref{Bramson-psi-estimate-for-one-sided-f} and $0<\alpha<\beta<\infty$. Then, there exist $t_0\ge 1, C_1, C_2 >0$ such that for all $t \ge t_0$ and $z \in [\alpha t, \beta t]$,
    \begin{equation}\label{quant-bound-for-super-critical}
        C_1\frac{e^{-\sqrt 2 z}}{\sqrt t}e^{-z^2/2t} \le u_\phi(t,z+\sqrt 2t)\le C_2\frac{e^{-\sqrt 2 z}}{\sqrt t}e^{-z^2/2t}.
    \end{equation}
The upper bound in the above holds for $\phi$ with $f\in C_c^+(\mathbb{R})$ as well. 
    \item Let $0< \alpha < \beta < \infty$. Then, there exists $t_0 \geq 1$, $K>0$ large enough and constants $C_1, C_2 >0$ such that for all $t\geq t_0$ and $z\in [\alpha t, \beta t]$,
    \begin{equation}\label{quant-bound-for-super-critical-for-difference}
        C_1\frac{e^{-\sqrt 2 z}}{\sqrt t}e^{-z^2/2t} \le u_M(t, \sqrt{2}t + z - K ) - u_M(t, \sqrt{2}t + z + K )\le C_2\frac{e^{-\sqrt 2 z}}{\sqrt t}e^{-z^2/2t}.
    \end{equation}
\end{enumerate}
    
\end{lemma}

\begin{proof}
The proof of part-$(a)$ follows easily by the same arguments as the proof of [Proposition $3.1$-\cite{bovier2014extremal}], which again is a direct consequence of Bramson's $\psi$ estimate \eqref{Bramson-psi-estimate-for-one-sided-f}. We leave the details to the reader. The upper bound \textcolor{black}{for $\phi$ with} $f\in C_c^+(\mathbb{R})$ follows easily by writing $f(x) \leq C 1_{[b,\infty)}(x)=: \widetilde{f}(x)$ for some $b\in \mathbb{R}$ and applying comparison principle $u_\phi \leq u_{\widetilde{\phi}}$ with $\widetilde{\phi}(x) =1-e^{-\widetilde{f}(-x)}$. \\
For part-$(b)$, the upper bound is immediate from \eqref{quant-bound-for-super-critical}. For the lower bound, we claim that for $K>0$ and $t>0$ large enough, 
     \begin{equation}\label{Eqn_ratio of two F-KPP}
         \frac{u_M(t,\sqrt 2t+z-K)}{u_M(t,\sqrt 2t+z+K)} \ge 2
     \end{equation}
for all $z\in [\alpha t, \beta t]$. Using \eqref{quant-bound-for-super-critical}, this implies that
     \begin{align*}
          u_M(t,\sqrt 2t+z-K)-u_M(t,\sqrt 2t+z+K) &\ge u_M(t,\sqrt 2t+z+K) \\
          & \gtrsim \frac{e^{-\sqrt 2 z}}{\sqrt t}e^{-z^2/2t}.
     \end{align*}
     It remains to verify \eqref{Eqn_ratio of two F-KPP}.
     To this end, using \eqref{quant-bound-for-super-critical} again, we obtain that
     \begin{align*}
          \frac{u_M(t,\sqrt 2t+z-K)}{u_M(t,\sqrt 2t+z + K)} &\gtrsim \frac{e^{-\sqrt 2(z-K)}e^{-\frac{(z-K)^2}{2t}}}{e^{-\sqrt 2(z+K)}e^{-\frac{(z+K)^2}{2t}}} \\
          & \gtrsim e^{2\sqrt 2K+\frac{2zK}{t}} \\& 
          \gtrsim e^{2\sqrt 2K}.
     \end{align*}
     Hence, \eqref{Eqn_ratio of two F-KPP} follows by choosing $K$ large enough. 
\end{proof}

Similarly as \eqref{C(a)-specific-to-lambda} which holds for $\lambda >\sqrt{2}$, we have the following lemma for $\lambda = \sqrt{2}$:

\begin{lemma}[Proposition 7.9 and Lemma 7.5 of \cite{Bovier_book}] \label{C-for-critical-def}
For $f$ of the form \eqref{Form of f}, 
\begin{equation}\label{constant-f}
\sqrt{\frac{2}{\pi}}\int_0^\infty u_\phi(t, x+\sqrt{2}t)xe^{\sqrt{2} x} dx \to \Cb(f)
\end{equation}
as $t\to \infty$, where $\Cb(f)$ is a positive constant. Furthermore, for any fixed $x \in \mathbb{R}$, 
\begin{equation}\label{C-shift}
\Cb(f(x+ \cdot)) = e^{\sqrt{2}x}\Cb(f).
\end{equation} 
\end{lemma}

If $f$ is such that $\phi (x) = 1-e^{-f(-x)} =  1_{(-\infty, 0]}(x)$, motivated by the relation \eqref{F-KPP-Formula-For-M}, we write

\begin{equation} \label{rename-C(f)-for-M}
    \Cb(f) = \Cb_M.
\end{equation}

\vspace{3mm}

\begin{lemma}[Lemma 4.7 of \cite{ABK13}]\label{Lemma upper and lower of u}
There exists a constant $C_1>0$ such that for all $t$ large enough and  $z\geq -\frac{1}{2} \log t $, 
\begin{equation}\label{upper-bound-for-critical-U-M-for-x>-log t}
u_M(t,z+\sqrt 2t) \le C_1 t^{-\frac{3}{2}}(z+\log t)e^{-\sqrt 2 z} e^{-\frac{z^2}{2t}}. 
\end{equation}
\end{lemma}

The following lemma also allows us to estimate $u_\phi$ from $u_M$ using \eqref{upper-bound-for-critical-U-M-for-x>-log t}.

\begin{lemma}\label{when-does-u-go-to-zero}
Let $f:\mathbb{R}\to  [0,\infty)$ be a non-negative measurable function supported in $[-K, \infty)$ and $\phi(x) = 1-e^{-f(-x)}$. Then, 
\begin{equation}\label{upper-bound-for-F-KPP-using-many-to-one}
    u_\phi(t, \lambda t -x) \leq e^t \mathbb{E}[f(x+ B_t - \lambda t)],
\end{equation}
and 
\begin{equation}\label{u-phi-bound-using-u-M}
     u_\phi(t, \lambda t -x) \leq u_M(t, \lambda t -x -K).
\end{equation}

In particular, for $f$ either compactly supported or of the type \eqref{Form of f}, 
\begin{equation}\label{convegence-zero-for-x>-log(t)}
u_\phi(t, \sqrt{2}t - x) \to 0    
\end{equation}
uniformly over $x\leq c\log t $ as $t\to \infty$ for some small enough constant $c>0$. 
\end{lemma}

\begin{proof}
    Using \eqref{Laplace-comp-using-F-KPP-Step-1}, the many-to-one Lemma \ref{Lemma many to one}, and the fact that $1-e^{-y} \leq y \wedge 1$ for $y\geq 0$, 
\begin{align} 
    \nonumber u_\phi(t, \lambda t -x) &= \mathbb{E}[1 -e^{-\sum_{k=1}^{n(t)} f(x + \chi_k(t) -\lambda t )}] \\
    \nonumber &= \mathbb{E}[(1 -e^{-\sum_{k=1}^{n(t)} f(x + \chi_k(t) -\lambda t )})1_{x+ M_t -\lambda t \geq -K}]\\
    \nonumber &\leq \mathbb{E}[\min \bigl\{1, \sum_{k=1}^{n(t)} f(x + \chi_k(t) -\lambda t )\bigr\} 1_{x + M_t -\lambda t \geq -K}] \\
    \nonumber & \leq \min \bigl\{e^t \mathbb{E}[f(x+ B_t - \lambda t)], \mathbb{P}(x+ M_t -\lambda t \geq -K)\bigr\}.
\end{align}

The claim \eqref{convegence-zero-for-x>-log(t)} follows easily using \eqref{u-phi-bound-using-u-M} and \eqref{upper-bound-for-critical-U-M-for-x>-log t}. 

\end{proof}

\begin{lemma}[Lemma 2.8 of \cite{Shekhar_Domain}\footnote{The Lemma 2.8 in \cite{Shekhar_Domain} is stated somewhat differently as compared to \eqref{lower-bound-for-critical-U-M} here. However, the proof of Lemma 2.8 in \cite{Shekhar_Domain} in fact gives the estimate \eqref{lower-bound-for-critical-U-M}. We leave the details to the reader.}] \label{critical-case-lower-bound}

There exist constants $r, C_2 >0$ such that for all $t$ large enough and $z\geq -\frac{1}{2} \log t $, 
    \begin{equation}\label{lower-bound-for-critical-U-M}
u_M(t,z+\sqrt 2t) \ge C_2 t^{-\frac{3}{2}} (z+\log t)e^{-\sqrt 2 z}e^{-\frac{z^2}{2(t-r)}}.
\end{equation}
\end{lemma}

As a consequence of the above bounds, the following variant of \eqref{quant-bound-for-super-critical-for-difference} holds as well. 

\begin{lemma}\label{Lemma-on-difference-U_M-for-critical}

Given $b>0$, there exists $K, C_1, C_2>0$\footnote{We stress that we are here and in \eqref{quant-bound-for-super-critical-for-difference} writing ``there exists a $K>0$" rather than writing ``for all $K>0$". We expect \eqref{quant-bound-for-super-critical-for-difference}, \eqref{quant-bound-for-difference-critical} to hold for all $K>0$ rather than only for $K$ large enough. We were able to verify \eqref{quant-bound-for-super-critical-for-difference}, \eqref{quant-bound-for-difference-critical} only for large enough $K>0$, which will be sufficient for our purposes.} such that for all $t$ large enough and $z\in [-\frac{1}{4}\log t , b t]$,
    \begin{align}\label{quant-bound-for-difference-critical}
        C_1 t^{-\frac{3}{2}} (z+\log t)e^{-\sqrt 2 z}e^{-\frac{z^2}{2t}}\le u_M(t,  &\sqrt{2}t + z - K ) - u_M(t, \sqrt{2}t + z + K ) \\ \nonumber & \le C_2 t^{-\frac{3}{2}} (z+\log t)e^{-\sqrt 2 z}e^{-\frac{z^2}{2t}}.
    \end{align}
    
\end{lemma}

\begin{proof}
    The upper bound is immediate from \eqref{upper-bound-for-critical-U-M-for-x>-log t} by noting that for $z\in [-\frac{1}{4}\log t , b t]$, $e^{-(z-K)^2/2t} = O(e^{-z^2/2t})$.\\
    For the lower bound, proceeding similarly to Lemma \ref{lemma_upper and lower for F-KPP}, using \eqref{upper-bound-for-critical-U-M-for-x>-log t} and \eqref{lower-bound-for-critical-U-M},
    \begin{align*}
          \frac{u_M(t,\sqrt 2t+z-K)}{u_M(t,\sqrt 2t+z + K)} &\gtrsim \frac{t^{-\frac{3}{2}} (z -K +\log t)e^{-\sqrt 2 (z-K)}e^{-\frac{(z-K)^2}{2(t-r)}}}{t^{-\frac{3}{2}} (z+ K + \log t)e^{-\sqrt 2(z+K)}e^{-\frac{(z+K)^2}{2t}}} \\
          & \gtrsim e^{2\sqrt 2K} e^{\frac{(z+K)^2}{2t}- \frac{(z-K)^2}{2(t-r)}} \\& 
          \gtrsim e^{2\sqrt 2K}, 
     \end{align*}
     where in the last line we have used that $e^{\frac{(z+K)^2}{2t}- \frac{(z-K)^2}{2(t-r)}} \gtrsim 1 $ which holds clearly since $z\in [-\frac{1}{4}\log t , b t]$. Hence, by choosing $K$ large enough, we obtain 
     \[ \frac{u_M(t,\sqrt 2t+z-K)}{u_M(t,\sqrt 2t+z + K)} \geq 2,\]
     which in turn implies 
     \[u_M(t,\sqrt 2t+z-K) - u_M(t,\sqrt 2t+z + K) \geq u_M(t,\sqrt 2t+z + K).\]

    Then, using \eqref{lower-bound-for-critical-U-M} once again completes the proof of the lower bound in \eqref{quant-bound-for-difference-critical}.
\end{proof}

\begin{remark}
    The lower bounds in estimates \eqref{quant-bound-for-super-critical-for-difference} and \eqref{quant-bound-for-difference-critical} required a careful verification, as presented above since Bramson's $\psi$ estimate (Proposition \ref{Bramson-psi-estimate-for-one-sided-f}) is not applicable for compactly supported initial conditions $\phi$. Proposition \ref{Bramson-psi-estimate-for-one-sided-f} is valid only for initial conditions $\phi$ satisfying certain conditions which are in particular satisfied by $\phi(x) = 1-e^{-f(-x)}$ with $f$ of the form \eqref{Form of f}, see [Chapter 6-\cite{Bovier_book}] for details.
\end{remark}

\subsection{Extremal point processes of BBM}
\label{Subsection Extremal Process}

For a binary  BBM $(\{\chi_k(t); 1\leq k\leq n(t)\}, t\ge0)$ on the real line, recall that we denote the maximal position at time $t$ by
\[
M_t:=\max_{1\leq k\leq n(t)}\chi_k(t). 
\]
Bramson \cite{Bramson78}, \cite{Bramson83} proved that $M_t-m(t)$ converges in law to some non-degenerate random variable $M_\infty$, where 
\begin{align}\label{e.m}
m(t):= \sqrt{2}t - \frac{3}{2\sqrt{2}}\log_{+}(t).
\end{align}
Then, Lalley and Sellke showed in \cite{Lalley_Sellke} that the limiting distribution function is a randomly shifted Gumbel distribution given by 
\begin{equation}\label{limiting-distribution-by-derivative-martingale}
    \mathbb{P}(M_\infty \leq x) =\mathbb{E}[e^{-\Cb_MZ_\infty e^{-\sqrt{2}x}}]
\end{equation} 
where $\Cb_M$ is as defined by \eqref{rename-C(f)-for-M} and $Z_\infty $ is the positive random variable which is the a.s. limit of the so-called \textit{derivative martingale}
\[Z_t := \sum_{k=1}^{n(t)}(\sqrt{2}t- \chi_k(t))\exp\bigl\{-\sqrt{2}(\sqrt{2}t- \chi_k(t))\bigr\}.\]
Later, it was proven in \cite{ABK13} and \cite{ABBS13} that the point process defined by
\begin{equation}\label{abk-conv}
    \mathcal{E}_t := \sum_{k=1}^{n(t)}\delta_{\chi_k(t)- m(t)} 
\end{equation}
converges in law to a non-trivial point process $\mathcal{E}_{\infty}$ as $t\to \infty$. The point process $\mathcal{E}_\infty$ is called the \textit{(limiting) extremal point process of BBM}. The law of $\mathcal{E}_\infty$ can be described as follows. 

Let $\mathcal{P}=\sum_{i\ge1}\delta_{p_i}$ be a Poisson point process independent of $Z_\infty$ and with intensity $\sqrt{2}\Cb_Me^{-\sqrt{2}x}dx$, where $\Cb_M$ is the same constant appearing in \eqref{limiting-distribution-by-derivative-martingale}. For each atom $p_i$ of $\mathcal{P}$, we attach a point process $\mathcal{D}^i=\sum_{j\ge1}\delta_{\mathcal{D}_j^i}$ where $\mathcal{D}^i$, $i\ge1$ are i.i.d. copies of a certain point process $\mathcal{D}$ and independent of $(\mathcal{P},Z_\infty)$. In this way, we get
\begin{equation}\label{def-Et}
 \mathcal{E}_{\infty} = \sum_{i, j} \delta_{p_i + \mathcal{D}^i_j + \frac{1}{\sqrt{2}}\log(Z_\infty)}.
 \end{equation}
The point process $\mathcal{E}_\infty$ is thus called a decorated Poisson point process with decoration process $\mathcal{D}$. Moreover, the decoration process $\mathcal{D}$ is a point process supported on $(-\infty,0]$ with an atom at $0$. 
The papers \cite{ABK13},\cite{ABBS13} also described its precise law as 

\begin{equation}\label{critical-deco-def}
    \mathbb P(\mathcal D\in \cdot) = \lim_{t\to\infty}\mathbb P(\sum_{k = 1}^{n(t)}\delta_{\chi_{k}(t) - M_{t} } \in \cdot | M_{t} \ge \sqrt{2}t).
\end{equation}

It is convenient to remove the randomness coming from the derivative martingale limit $Z_\infty$. Define 
\begin{equation}\label{tilde-def}
\widetilde{\mathcal{E}}_\infty := \mathcal{E}_\infty (\cdot + \frac{1}{\sqrt{2}}\log(Z_\infty))=\sum_{i, j} \delta_{p_i + \mathcal{D}^i_j}.
\end{equation} 

This process was proved in \cite{ABBS13} (also in \cite{ABK13}, but it is not explicitly stated there) to be the limit in law of 
\[\widetilde{\mathcal{E}}_{t} := \sum_{k=1}^{n(t)}\delta_{\chi_k(t)- m(t)- \frac{1}{\sqrt{2}}\log(Z_t)}.\]

Furthermore, it was proven in \cite{ABK13} that $\theta_t^{\sqrt{2}}$ given by \eqref{theta_t-def} for $\lambda = \sqrt{2}$ started from $\theta = PPP(\sqrt{\frac{2}{\pi}}(-x)e^{-\sqrt{2}x}1_{x < 0} dx)$ converges in distribution to $\widetilde{\mathcal{E}}_\infty$. Then, using the Laplace transform formula for PPP (see Proposition 2.12 of \cite{Bovier_book}), formula \eqref{Laplace-comp-using-F-KPP}, and \eqref{constant-f}, it follows that for all $f$ of type \eqref{Form of f}, 
\begin{equation}\label{critical-E-Laplace-fromula-using-C(f)}
\mathbb{E}[e^{-\langle f, \widetilde{\mathcal E}_\infty \rangle}] = e^{-\Cb(f) }. 
\end{equation}

Also, if $S$ is $[-\infty, \infty)$-valued random variable independent of $\widetilde{\mathcal E}_\infty $, using \eqref{C-shift},  \eqref{critical-E-Laplace-fromula-using-C(f)} implies that 

\begin{equation}\label{critical-E-Laplace-fromula-with-shift}
    \mathbb{E}[e^{-\langle f, \widetilde{\mathcal E}_\infty(\cdot -S)  \rangle}] = \mathbb{E}[e^{- e^{\sqrt{2} S}\Cb(f)}].
\end{equation}

\vspace{3mm}

For $\lambda > \sqrt{2}$, Kabluchko in \cite{Kabluchko} proved that $\theta_t^\lambda$, given by \eqref{theta_t-def} and started from 
$\theta = PPP(\frac{1}{\sqrt{2\pi}}e^{-\mu x } dx)$, converges in distribution to a non-trivial point process which we write as $\widetilde{\mathcal E}_\infty^\lambda$. 
Again, using the Laplace transform formula for PPP (see Proposition 2.12 of \cite{Bovier_book}), formula \eqref{Laplace-comp-using-F-KPP}, and \eqref{C(a)-specific-to-lambda}, it follows that for all $f$ of type \eqref{Form of f}, 
\begin{equation}\label{supercritical-E-Laplace-fromula}
\mathbb{E}[e^{-\langle f, \widetilde{\mathcal E}_\infty^\lambda  \rangle}] = e^{-C(f, \alpha(\lambda))}, 
\end{equation}
where $C(f, \alpha(\lambda))$ is given as in Lemma \ref{C(a) definition}. Furthermore, if $S$ is $[-\infty, \infty)$-valued random variable independent of $\widetilde{\mathcal E}_\infty^\lambda $, using \eqref{shift-relation-for-C(a)},  \eqref{supercritical-E-Laplace-fromula} implies that 

\begin{equation}\label{supercritical-E-Laplace-fromula-with-shift}
    \mathbb{E}[e^{-\langle f, \widetilde{\mathcal E}_\infty^\lambda(\cdot -S)  \rangle}] = \mathbb{E}[e^{- e^{\mu S}C(f, \alpha(\lambda))}].
\end{equation}

\vspace{5mm}

The point process $\widetilde{\mathcal E}_\infty^\lambda$ is a form of extremal point process, namely it is the supercritical extremal point process of BBM. As an extension to \eqref{critical-deco-def}, Bovier and Hartung introduced in \cite{bovier2014extremal} a family of decoration point processes $\mathcal{D}^\rho$ for $\rho \geq \sqrt{2}$ defined by 

\begin{equation}\label{decoration_family}
     \mathbb P(\mathcal{D}^{\rho} \in \cdot) := \lim_{t\to\infty}\mathbb P(\sum_{k = 1}^{n(t)}\delta_{\chi_{k}(t) - M_{t} } \in \cdot | M_{t} \ge \rho t).
\end{equation}

These decoration point processes $(\mathcal{D}^{\rho})_{\rho \ge \sqrt{2}}$ also appear in the extremal processes of variable speed branching Brownian motions \cite{bovier2015variable} and multi-type branching Brownian motions \cite{belloum2021anomalous}. The special case of $\rho = \sqrt{2}$ gives the decoration point process $\mathcal{D}^{\sqrt{2}} = \mathcal{D}$ of the critical extremal point process $\widetilde{\mathcal{E}}_\infty$ appearing in \eqref{tilde-def}. The $\mathcal{D}^\rho$ for $\rho >\sqrt{2}$ are called the supercritical decorations. The $\widetilde{\mathcal E}_\infty^\lambda$ is also a decorated Poisson point process with the decoration given by $\mathcal{D}^\mu$ for $\mu = \lambda + \sqrt{\lambda^2 -2} >\sqrt{2}$. More precisely, let $\mathcal{P}^\mu = \sum_{i\geq 1}\delta_{p_i}$ be a Poisson point process with intensity $\mu C_M(\alpha(\lambda))e^{-\mu x}dx$, where $C_M(\alpha(\lambda))$ is given by \eqref{rename-C(f)-for-M-for-supercritical}, and $\{\mathcal{D}^{\mu, i}\}_{i\geq 1}$ be an i.i.d. collection of point processes distributed as $\mathcal{D}^{\mu}$. Then, similarly to \eqref{tilde-def}, 

\begin{equation}\label{decoration-decomp-for-super-critical}
    \widetilde{\mathcal{E}}_\infty^\lambda =\sum_{i, j} \delta_{p_i + \mathcal{D}^{\mu,i}_j}.
\end{equation}

Since we could not find the above description of $\widetilde{\mathcal{E}}_\infty^\lambda$ explicitly stated in the literature, we include a proof of \eqref{decoration-decomp-for-super-critical} in Appendix \ref{Section Appendix}, see Proposition \ref{Identify_DPPP}.

\subsection{An excerpt from results in \cite{Shekhar_Domain}.}

\begin{proposition}[Theorem 1.4, Remark 6, Proposition 4.1 of \cite{Shekhar_Domain}]\label{excerpt-from-CGS-AAP}
Let $\theta$ be a point process such that $\theta \in \mathcal{N}_2$ and $\theta$  has a finite top particle almost surely. Suppose 
\[ Y_t := \frac{1}{t^{3/2}}\int_{-\infty}^0(-x)e^{\sqrt{2}x} e^{-\frac{x^2}{2t}} \theta(dx) \]
is tight in $t>0$. Then, given a sequence $\{t_n\}_{n\geq 1}$ and a non-negative random variable $Z$ such that $Y_{t_n} \overset{d}{\to} Z$, it holds that  
\begin{equation}\label{main-result-CGS-along-subsequence}
     \langle f, \theta_{t_n}^{\sqrt{2}} \rangle \overset{d}{\to} \langle f, \widetilde{\mathcal{E}}_\infty(\cdot - S)\rangle,
\end{equation}
for all $f$ of form \eqref{Form of f}, where $\theta_t^{\sqrt{2}}$ is given by \eqref{theta_t-def} for $\lambda =\sqrt{2}$ started from $\theta_0^{\sqrt{2}} = \theta $ and $S= \log Z/\sqrt{2}$. In particular, \eqref{main-result-CGS-along-subsequence} holds for all $f\in C_c^+(\mathbb{R})$ as well.  
\end{proposition}

\section{Some Derived properties of $\theta \in Dom_{\lambda}(\Pi)$.} \label{Section Properties theta in Don}
\textcolor{black}{
In this section, we derive several structural properties of the point process
$\theta \in Dom_{\lambda}(\Pi)$. Although $\theta$ is assumed to be only locally finite, the condition $\theta \in Dom_{\lambda}(\Pi)$
forces $\theta$ to satisfy certain properties, including integrability, tightness estimates, and lacunarity. These properties and estimates will be key
ingredients in the proofs of the main theorems in the subsequent sections.
}

\subsection{Integrability and tightness estimates for $\theta \in Dom_{\lambda}(\Pi)$, \textcolor{black}{with $\lambda \ge \sqrt 2$}}

\begin{proposition}\label{crucial-tightness-prop}
    Let $\Pi$ be a fixed point of BBM with drift $\lambda \geq \sqrt{2}$ and let $\theta \in Dom_{\lambda}(\Pi)$ be a point process. Then, 
    \begin{enumerate}
        \item $\theta \in \mathcal{N}_2$ almost surely (recall the space $\mathcal{N}_2$ given by \eqref{N_2-def}). 

        \item For each $K>0$, the family of random variables $\{Z_t^a\}_{t > 0}$ given by \[ Z_t^a = \int_{-\infty}^{+\infty} \mathbb{P}(|x + B_t - \lambda t| \leq K) \theta(dx)\] is tight, where $B_t$ is a standard Brownian motion.  

        \item For each $K>0$, the family of random variables $\{Z_t^b\}_{t > 0}$ given by \[ Z_t^b = \int_{-\infty}^{+\infty} \mathbb{P}(|x + M_t - \lambda t| \leq K) \theta(dx)\] is tight, where $M_t$ is the maximum particle of BBM $\{\chi_k(t)\}_{k\leq n(t)}$. 
    \end{enumerate}
\end{proposition}

\begin{proof}
    In order to prove part-$(1)$, we first claim that for all $K >0$,
\begin{equation}\label{Eqn No blow up}
        \mathbb{P}\left(\int_{-\infty}^{\infty}\mathbb{P}(|x+B_{t}-\lambda t|\le K)\theta(dx) = +\infty  \right) \longrightarrow 0 \quad \text{as}\quad t\to\infty.
    \end{equation}
    
To this end, let $f \in C^{+}_{c}(\mathbb{R})$ such that $f(x)\equiv 1$ for $x \in [-K,K]$ and $f(x) = 0$ for $|x| \geq 2K$. Since $\theta \in $ $Dom_{\lambda}(\Pi)$, we have, for all $\rho>0$, as $t\to\infty$,
    \begin{equation}\label{apply-domain-condition}
        \mathbb{E}\left[e^{-\rho\langle f,\theta_t^\lambda \rangle} \right]\longrightarrow \mathbb{E}\left[e^{-\rho\langle f,\Pi \rangle} \right].
    \end{equation}
    Note that it is a priori possible that $\langle f,\theta_t^\lambda \rangle = +\infty$ with positive probability. But, since $\Pi$ is locally finite, we have $\langle f,\Pi \rangle<\infty$ almost surely. Then, by Lemma~\ref{Feller's lemma},
\[       \mathbb{P}\left(\langle f,\theta_t^\lambda \rangle = +\infty \right)\longrightarrow 0  \textrm{ as } t \to \infty.
\]

This, in turn, easily implies that 
\[\mathbb{P}\left(\theta^\lambda_t([-K,K]) = +\infty \right) \to 0 \textrm{ as } t \to \infty,\]
and 
\begin{equation}\label{U-thinning-zero}
        \mathbb{P}\left(U(\theta^\lambda_t)([-K,K]) = +\infty \right) \longrightarrow 0 \quad \text{as}\quad t \to\infty,
    \end{equation}
    
where $U(\theta_t^\lambda)$ is the thinning of $\theta^\lambda_t$ given by \eqref{def-thinning-U}. Note that 
    \begin{equation}\label{U-conditional-expectation}
        \mathbb{E}\left[U(\theta^\lambda_t)([-K,K]) \hspace{1mm}\bigl| \hspace{1mm}\theta \right] = \int_{-\infty}^{\infty}\mathbb{P}(|x+B_{t}-\lambda t| \le K)\theta(dx).
    \end{equation}

    Conditional on $\theta$, $U(\theta^\lambda_t)([-K,K])$ is a sum of independent Bernoulli random variables. Hence, by part $(b)$ of Lemma \ref{Lemma_bernoulli}, the event $\{ \int_{-\infty}^{\infty}\mathbb{P}(|x+B_{t}-\lambda t| \le K)\theta(dx) = +\infty\} $ implies $\{U(\theta^\lambda_t)([-K,K]) = +\infty\}$, that is,
     \begin{equation}
        \mathbb{P}\left(\int_{-\infty}^{\infty}\mathbb{P}(|x+B_{t}-\lambda t|\le K)\theta(dx) = +\infty  \right) \le \mathbb{P}\left( U(\theta_t^\lambda)([-K,K]) = +\infty \right).
    \end{equation}
The above implies \eqref{Eqn No blow up} using \eqref{U-thinning-zero}.

    Now in order to conclude that $\theta \in \mathcal{N}_2$ almost surely, we note that for a fixed $t$, as $|x|\to\infty$, 
      
      \begin{equation}
        e^{\frac{x^2}{t}}\mathbb{P}\left(|x+B_t-\lambda t|\le K \right) \longrightarrow \infty.
    \end{equation}
    Therefore, since $\theta$ is locally finite, $\int_{-\infty}^{\infty} e^{-\frac{x^2}{t}}\theta(dx) = +\infty$ implies $\int_{-\infty}^{\infty}\mathbb{P}(|x+B_t-\lambda t|\le K)\theta(dx) = +\infty$, which means
    \begin{equation}
        \mathbb{P}\left( \int_{-\infty}^{\infty}e^{-\frac{x^2}{t}}\theta(dx) = +\infty \right) \le \mathbb{P}\left( \int_{-\infty}^{\infty}\mathbb{P}(|x+B_t-\lambda t|\le K)\theta(dx) = +\infty \right).
    \end{equation}
    Hence, from \eqref{Eqn No blow up}, 
    \begin{equation}\label{no-blow-up-2}
        \mathbb{P}\left( \int_{-\infty}^{\infty}e^{-\frac{x^2}{t}}\theta(dx) = +\infty \right)\longrightarrow 0 \quad \text{as}\quad t\to\infty.
    \end{equation}

    Let $A_{n}:=\left\{\int_{-\infty}^{\infty}e^{-\frac{x^2}{n}}\theta(dx) < \infty \right\}$. As the events $\{A_{n}\}_{n\in\mathbb{N}}$ are decreasing, we infer from \eqref{no-blow-up-2} that $\mathbb{P}(A_{n}) \to 1$ as $n\to\infty$. Hence
    \begin{equation}
        \mathbb{P}\left(\int_{-\infty}^{\infty}e^{-\alpha x^2}\theta(dx) <\infty ,\forall \alpha >0\right) = \mathbb{P}\left( \bigcap_{n\ge1}A_{n}\right) = \lim_{n\to\infty}\mathbb{P}(A_{n}) = 1,
    \end{equation}
    which finishes the proof of the claim $\theta \in \mathcal{N}_2$ almost surely. \\

    The proof of part (2) follows the argument in Section 3.4 of \cite{Shekhar_FP}. For completeness, we include the short proof. 
    \textcolor{black}{Now $\theta \in \mathcal N_2$ a.s.} implies that $Z_t^a$ is almost surely finite and hence a well-defined real-valued random variable. Furthermore, it follows easily by continuity of $Z_t^a$ that for all $T>1$, $\sup_{t\in [1,T]} |Z_t^a| <\infty$ almost surely. Hence, it suffices to prove the tightness of $Z_t^a$ for $t>T$ with sufficiently large $T$. Also, from \eqref{N_2-is-fixed}, it follows that $\theta_t^\lambda$ is a well-defined point process. In particular, the thinning $U(\theta_t^\lambda)$ given by \eqref{def-thinning-U} is a well-defined point process. Hence, for $f \in C_c^+(\mathbb R)$ such that $f(x)=1$ for $x\in [-K,K]$ and $f(x)=0$ for $|x|\ge 2K$, $\langle f, \theta_t^\lambda \rangle$ is finite almost surely. Using \eqref{apply-domain-condition} again, we conclude that $\langle f, \theta_t^\lambda \rangle$ converges in distribution to $\langle f, \Pi \rangle$. In particular, $\langle f, \theta_t^\lambda \rangle$ is tight. Let us write $\mathcal{Z}_t^a = U(\theta_t^\lambda)([-K, K])$. Since $\mathcal{Z}_t^a \leq \langle f, \theta_t^\lambda \rangle$, it follows that $\mathcal{Z}_t^a$ is tight. Recall \eqref{U-conditional-expectation}, i.e. $Z_t^a = \mathbb{E}[\mathcal{Z}_t^a \hspace{0.8mm}\bigl | \hspace{0.8mm} \theta]$, and conditional on $\theta$, $\mathcal{Z}_t^a$ is a sum of independent Bernoulli random variables. The tightness of $Z_t^a$ follows by noting
\begin{align}
\mathbb{P}(Z_t^a\geq L)\leq & \mathbb{P}\left(Z_t^a\geq L; |\mathcal{Z}_t^a-Z_t^a|\geq \frac12 Z_t^a\right)+\mathbb{P}\left(Z_t^a\geq L; \mathcal{Z}_t^a\ge\frac12 Z_t^a\right)\nonumber\\
\leq & \mathbb{E}\left[\frac{4}{Z_t^a}1_{Z_t^a\geq L}\right]+\mathbb{P}\left(\mathcal{Z}_t^a\geq \frac12 L\right) \nonumber \\
\leq & \frac{4}{L} + \mathbb{P}\left(\mathcal{Z}_t^a\geq \frac12 L\right),\nonumber
\end{align}
where we applied part-$(a)$ of Lemma \ref{Lemma_bernoulli} conditional on $\theta$ in first line to second line in the above. Since $\mathcal{Z}_t^a$ is tight, the above computation implies that $Z_t^a$ is tight as well. \\

The proof of the part-$(3)$ is similar to the above: we apply the above arguments to the thinning $M(\theta_t^\lambda)$ given by \eqref{def-thinning-M} instead of $U(\theta_t^\lambda)$. 
\end{proof}

\begin{remark}
    The tightness of $Z_t^b$ is deemed to be a better estimate than the tightness of $Z_t^a$ since we expect $\mathbb{P}(|x + M_t - \lambda t| \leq K) \approx e^t \mathbb{P}(|x + B_t - \lambda t| \leq K)$. However, in view of Proposition \ref{Bramson-psi-estimate-for-one-sided-f}, we could only verify this for a certain range of $x$, e.g. for $\lambda>\sqrt{2}$ and $x\in [-bt, -at]$, $a,b>0$. In particular, we could not verify this for $x$ in a neighbourhood of $+\infty$, e.g. for $\{x >bt\}$. As such, the tightness of $Z_t^a$ will also be useful to us.  
\end{remark}

\subsection{Lacunary property of $\theta \in Dom_{\lambda}(\Pi)$}

\begin{proposition}\label{lacunary-prop}
 Let $\Pi$ be a fixed point of BBM with drift $\lambda \geq \sqrt{2}$ and let $\theta \in Dom_{\lambda}(\Pi)$ be a point process. Then, (recall Definition \ref{def-lacunary})
 \begin{enumerate}
     \item For $\lambda > \sqrt{2}$, $\theta$ is lacunary.
     \item For $\lambda = \sqrt{2}$, $\theta$ is exponentially-lacunary. 
 \end{enumerate}
\end{proposition}

\begin{proof}
 Let $K>0$ and $f\in C_c^+(\mathbb{R})$ such that $f(x)=1$ for $x\in [-K,K]$ and $f(x)=0$ for $|x|\geq 2K$. Then, for $\rho >0$, using Lemma \ref{Formula-for-Laplace-of-theta_t}, 

 \begin{equation}\label{Laplace-comp-using-F-KPP-2}
\mathbb{E}[e^{-\rho \langle f,\theta_t^\lambda\rangle}] = \mathbb{E}\biggl[\exp\biggl\{\int_{-\infty}^\infty\log(1- u_\phi(t, \lambda t -x))\theta(dx)\biggr\}\biggr],
\end{equation}
where $\phi(x) = 1 - e^{-\rho f(-x)}$. For $\delta >0$, let 
\begin{equation}\label{def-Q_t-delta}
    Q_t^\delta := \{ x\in \mathbb{R} \hspace{1mm}\bigl | \hspace{1mm} |\lambda t -x| \leq \sqrt{2}t - \bigl(\frac{3}{2\sqrt{2}} + \delta\bigr)\log t\}.
\end{equation} 
Then, by using Lemma \ref{Bramson-converges-to-1}, 
\[\alpha_t := \sup_{x\in Q_t^\delta} (1 - u_\phi(t, \lambda t -x)) \to 0 \] 
as $t\to \infty$. Hence,
\begin{align}
    \mathbb{E}[e^{-\rho \langle f,\theta_t^\lambda\rangle}] &= \mathbb{E}\biggl[\exp\biggl\{\int_{-\infty}^\infty\log(1- u_\phi(t, \lambda t -x))\theta(dx)\biggr\}\biggr] \\ \nonumber 
    & \leq \mathbb{E}\biggl[\exp\biggl\{\int_{Q_t^\delta}\log(1- u_\phi(t, \lambda t -x))\theta(dx)\biggr\}\biggr] \\ \nonumber
    &\leq \mathbb{E}\biggl[\exp\biggl\{\log\alpha_t \times \theta(Q_t^\delta)\biggr\}\biggr] \\ \nonumber
    &= \mathbb{E}\biggl[\exp\biggl\{\log\alpha_t \times \theta(Q_t^\delta)\biggr\} 1_{\theta(Q_t^\delta) \geq 1}\biggr] + \mathbb{P}(\theta(Q_t^\delta) = 0)\\ \nonumber
    &\leq  \alpha_t + \mathbb{P}(\theta(Q_t^\delta) = 0).\\ \nonumber
\end{align}
Taking $t\to \infty$ in the above, since $\alpha_t \to 0$ and $\theta \in Dom_{\lambda}(\Pi)$, we obtain 
\[ \liminf_{t\to \infty} \mathbb{P}(\theta(Q_t^\delta) = 0) \geq  \mathbb{E}[e^{-\rho \langle f,\Pi\rangle}],\]
and since $\Pi$ is locally finite almost surely, taking $\rho \to 0^+$, we conclude that $\mathbb{P}(\theta(Q_t^\delta) = 0) \to 1$, which is equivalent to $\theta(Q_t^\delta) \overset{\mathbb{P}}{\to} 0$ as $t\to \infty$. Finally, we note that for $Q_t^\delta$ given by \eqref{def-Q_t-delta}, if $\lambda >\sqrt{2}$,  $c>\lambda - \sqrt{2}$ and $d<\lambda + \sqrt{2}$, then $[ct, dt] \subset Q_t^\delta $ for all $t$ large enough. For $\lambda = \sqrt{2}$, if $c > \frac{3}{2\sqrt{2}} + \delta$ and $d< 2\sqrt{2}$, then $[c\log t , dt] \subset Q_t^\delta$ for all $t$ large enough. Hence, $\theta([ct, dt]) \overset{\mathbb{P}}{\to} 0 $ (resp. $\theta([c\log t, dt]) \overset{\mathbb{P}}{\to} 0$) for $\lambda >\sqrt{2}$ (resp. $\lambda = \sqrt{2}$). 

\end{proof}

\section{Proof of Theorem \ref{main-thm} and Theorem \ref{main-thm-2} for $\lambda >\sqrt{2}$.} \label{Section proof supercritical}

\textcolor{black}{In this section, we prove Theorem \ref{main-thm} and Theorem \ref{main-thm-2} for the case $\lambda>\sqrt{2}$. Following a strategy similar to that in \cite{Shekhar_FP}, we decompose the contribution of the initial point process into several regions and show that only particles in an $O(\sqrt{t})$-neighbourhood of the location $-\sqrt{\lambda^2-2}t$ contribute to the limiting point process.}

\textcolor{black}{In particular, the structural properties established in Section \ref{Section Properties theta in Don} allow us to control the contributions from different regions without imposing the additional assumptions on the initial point process required in \cite{Shekhar_FP} and \cite{Shekhar_Domain}.}


Let $a, b, L>0$ be fixed constants such that $b>\lambda + \sqrt{2}$ and $a < \lambda - \sqrt{2}$. For $\mu = \lambda + \sqrt{\lambda^2 -2 }$, write $\gamma = \mu - \lambda = \sqrt{\lambda^2 -2}$. We divide $\mathbb{R}$ into the following sets: 
\begin{align}
    \nonumber E_t^1 &= (-at, at), \\ \nonumber
    E_t^2 &= (-\infty, -bt)\cup (bt, \infty), \\ \label{Def-sets-E} 
    E_t^3 &= [at,bt], \\ \nonumber
    E_t^4 &= [-bt, -\gamma  t -L\sqrt{t}], \\ \nonumber
    E_t^5 &= [-\gamma t + L\sqrt{t}, -at], \\ \nonumber
    E_t^6 &= (-\gamma t - L\sqrt{t}, -\gamma t + L\sqrt{t}).  
\end{align}

For a fixed point $\Pi$ and a point process $\theta \in Dom_{\lambda}(\Pi)$, for $f\in C_c^+(\mathbb{R})$, Lemma \ref{Formula-for-Laplace-of-theta_t} gives  

\begin{equation}\label{Laplace-comp-using-F-KPP-3}
\mathbb{E}[e^{-\langle f,\theta_t^\lambda\rangle}] = \mathbb{E}\biggl[\exp\biggl\{\int_{-\infty}^\infty\log(1- u_\phi(t, \lambda t -x))\theta(dx)\biggr\}\biggr],
\end{equation}
where $\phi(x) = 1-e^{-f(-x)}$. We write 
\begin{equation}
    \int_{-\infty}^\infty\log(1- u_\phi(t, \lambda t -x))\theta(dx) = \sum_{i=1}^6 I_t^i,
\end{equation}
where
\begin{equation}
   I_t^i :=  \int_{E_t^i}\log(1- u_\phi(t, \lambda t -x))\theta(dx).
\end{equation}

\vspace{3mm}

We show that in the limit $t \to \infty$,  $I_t^1, I_t^2, I_t^3 \overset{\mathbb{P}}{\to} 0 $. Also, $I_t^4, I_t^5 \overset{\mathbb{P}}{\to} 0$ as $t \to \infty$ and then $L\to \infty$. Finally, we show that $I_t^6$ will converge in distribution to a random variable as $t \to \infty$ and then $L\to \infty$.

\subsection{Proof of convergence $I_t^1, I_t^2 \overset{\mathbb{P}}{\to} 0$.}

We first record the following lemma, which will be used in the proof.

\begin{lemma}\label{Lemma Ratio 1}
Let $K>0$ and $E_t=E_t^1 \text{ or } E_t^2$. Then, there exists some $u \ge 0$ \textcolor{black}{(depending on $t$)} such that
    \begin{equation}\label{Eqn ratio 1}
        \sup_{x \in E_t} \frac{e^t \mathbb P(|x+B_t-\lambda t|\le K)}{\mathbb P(|x+B_u-\lambda u|\le K)} \longrightarrow 0
    \end{equation}
    as $t \to \infty.$ 
    In particular, 
    \begin{equation}\label{Eqn implies by ratio 1}
        \sup_{x\in E_t} e^t \mathbb P(|x+B_t-\lambda t|\le K) \to 0
    \end{equation}
    as $t \to \infty$. \eqref{Eqn ratio 1} and \eqref{Eqn implies by ratio 1} hold for $\lambda =\sqrt{2}$ and $E_t = E_{t}^2 = \{|x| > bt \}$ for any $b>2\sqrt{2}$ as well.

\end{lemma}

\begin{proof}
    By definition
    \begin{equation*}
        \begin{aligned}
            e^t \mathbb P(|x+B_t-\lambda t|\le K)=\int_{-K}^K \frac{e^t}{\sqrt{2\pi t}} e^{-\frac{(y-x+\lambda t)^2}{2t}}dy, \\
            \mathbb P(|x+B_u-\lambda u|\le K)=\int_{-K}^K \frac{1}{\sqrt{2\pi u}} e^{-\frac{(y-x+\lambda u)^2}{2u}}dy.
        \end{aligned}
    \end{equation*}
    We show that the ratio of the two densities above goes to zero uniformly over $x\in E_t$, $y\in[-K, K]$ for some choice of $u$, i.e.
    \begin{equation} \label{Eqn_limit in Lemma ratio 1}
        \sup_{x \in E_t, y\in [-K,K]} \sqrt{\frac{u}{t}} e^t e^{\frac{(y-x+\lambda u)^2}{2u}-\frac{(y-x+\lambda t)^2}{2t}} \longrightarrow 0
    \end{equation}
    as $t \to \infty.$
    Note that
    \begin{equation*}
        e^t e^{\frac{(y-x+\lambda u)^2}{2u}-\frac{(y-x+\lambda t)^2}{2t}}=e^{t+\frac{\lambda^2}{2}(u-t)+(y-x)^2\left[\frac{1}{2u}-\frac{1}{2t}\right]}.
    \end{equation*}

    For $E_t=E_t^1$, we pick $u=\frac{L_t}{\lambda}$ where $L_t:=\sup\limits_{x \in E_t^1, y \in [-K,K]} |y-x|$. Note that $u<t$ and 
    \begin{equation*}
        \begin{aligned}
            e^{t+\frac{\lambda^2}{2}(u-t)+(y-x)^2\left[\frac{1}{2u}-\frac{1}{2t}\right]} &\le e^{t+\frac{\lambda^2}{2}(u-t)+L_t^2 \left[\frac{1}{2u}-\frac{1}{2t}\right]} \\
            &= e^{t+\lambda L_t-\frac{L_t^2}{2t}-\frac{\lambda^2 t}{2}} \\
            &\leq e^{-\delta t},
        \end{aligned}
    \end{equation*}
   for some $\delta >0$, where in the last inequality we have used the fact that for $L_t = at +K $ for some $a <\lambda -\sqrt{2}$. This implies \eqref{Eqn_limit in Lemma ratio 1}. Similarly, for $E_t=E_t^{2}$, we note that the above estimate continues to hold by taking $u=\frac{\widetilde{L}_t}{\lambda} >t $ where ${\widetilde L}_t:=\inf\limits_{x \in E_t^2, y \in [-K,K]} |y-x| = bt-K $ for some $b>\lambda + \sqrt{2}$. Also note that this estimate holds for $\lambda = \sqrt{2}$ and $E_t = E_t^2 = \{|x| >bt\}$ for any $b>2\sqrt{2}$ as well.   
\end{proof}

Now, let $f\in C_c^+(\mathbb{R})$ and $K>0$ be such that $f(x) =0$ for $|x| \geq K$. Using \eqref{upper-bound-for-F-KPP-using-many-to-one}, we obtain that 

\begin{align}
    \nonumber u_\phi(t, \lambda t -x) &\leq  e^t \mathbb{E}[f(x+ B_t - \lambda t)] \\
     & \lesssim e^t \mathbb{P}(|x+ B_t - \lambda t| \leq K). \label{u<e^tP(B)-bound}
\end{align}
By \eqref{Eqn implies by ratio 1}, $u_\phi(t, \lambda t -x) \to 0 $ as $t\to \infty$ uniformly over $x\in E_t^1$ and $x\in E_t^2$. Hence, for $i=1,2$, 

\begin{equation}
    I_t^i =  \int_{E_t^i}\log(1- u_\phi(t, \lambda t -x))\theta(dx) = -(1+ o_t(1)) \int_{E_t^i}u_\phi(t, \lambda t -x)\theta(dx).
\end{equation}
Furthermore, using \eqref{u<e^tP(B)-bound}, 

\begin{align}
    \label{final-step-for-i=1,2} \int_{E_t^i}u_\phi(t, \lambda t -x)\theta(dx) &\lesssim \int_{E_t^i}e^t \mathbb{P}(|x+ B_t - \lambda t| \leq K) \theta(dx) \\
     \nonumber &\leq \sup_{x\in E_t^i} \frac{e^t \mathbb{P}(|x+ B_t - \lambda t| \leq K)}{\mathbb{P}(|x+ B_u - \lambda u| \leq K)} \int_{E_t^i}\mathbb{P}(|x+ B_u - \lambda u| \leq K) \theta(dx),
\end{align}
where we have picked $u$ from Lemma \ref{Lemma Ratio 1}. Then, from \eqref{Eqn ratio 1} and the tightness of $Z_t^a = \int_{-\infty}^{+\infty} \mathbb{P}(|x + B_t - \lambda t| \leq K) \theta(dx)$ given by Proposition \ref{crucial-tightness-prop}, we conclude that 
\[ \int_{E_t^i}u_\phi(t, \lambda t -x)\theta(dx) \overset{\mathbb{P}}{\to} 0,\]
which in turn implies $I_t^i \overset{\mathbb{P}}{\to} 0$ for $i=1,2$ as $t \to \infty$.

\subsection{Proof of convergence $I_t^3 \overset{\mathbb{P}}{\to} 0$.}

The proof follows immediately from Proposition \ref{lacunary-prop} as follows. Since $\theta$ is an integer-valued measure and \(I_t^3\leq0\),
\begin{align}
    \nonumber \mathbb P(I_t^3\neq0)=\mathbb P(I_t^3<0) &= \mathbb{P}\biggl( \int_{[at,bt]}\log(1- u_\phi(t, \lambda t -x))\theta(dx) <0\biggr) \\
    \nonumber &\leq \mathbb{P}(\theta([at, bt]) >0) \\
    \nonumber &= \mathbb{P}(\theta([at, bt]) \geq 1).
\end{align}

Then, since $\theta$ is lacunary, using Remark \ref{about-lacunary}, $\mathbb{P}(\theta([at, bt]) \geq 1) \to 0$. Consequently \(I_t^3\to0\) in probability.

\subsection{A reformulation of the tightness property of $\theta \in Dom_{\lambda}(\Pi)$ for $\lambda > \sqrt{2}$.}

Before proving $I_t^4, I_t^5 \overset{\mathbb{P}}{\to} 0$, we record a consequence of the tightness property $\theta \in Dom_{\lambda}(\Pi)$ given by Proposition \ref{crucial-tightness-prop}. 

\begin{proposition}\label{replace-M-by-explicit-expression-in-tightness}
    Let $\Pi$ be a fixed point of BBM with drift $\lambda > \sqrt{2}$ and $\theta \in Dom_{\lambda}(\Pi)$. Then, the family of random variables $Z_t$ given by 

\begin{equation}
  Z_t := \int_{- \infty}^0 \frac{1}{\sqrt{t}}e^{(1-\frac{\lambda^2}{2})t}e^{\lambda x}e^{-\frac{x^2}{2t}}\theta(dx) 
\end{equation}
is tight in $t\geq 1$. 
\end{proposition}

\begin{proof}
Using \eqref{final-step-for-i=1,2}, we have for $b>\lambda + \sqrt{2}$, 

\[ \int_{-\infty}^{-bt} e^t \mathbb{P}(|x + B_t - \lambda t| \leq K)\theta(dx)\overset{\mathbb{P}}{\to} 0\]
as $t\to \infty$. Also, uniformly over $x<0$,

\begin{align} \label{Brownian-in-negative-part-actually-gives-sharp-bound-super-critical}
    e^t \mathbb{P}(|x + B_t - \lambda t| \leq K) &= \frac{e^t}{\sqrt{2\pi t}} \int_{-K}^K e^{-\frac{(y + \lambda t -x)^2}{2t}}dy \\ 
     \nonumber &=  \frac{1}{\sqrt{2\pi t}}e^{(1-\frac{\lambda^2}{2})t}e^{\lambda x}e^{-\frac{x^2}{2t}} \int_{-K}^K e^{-\frac{y^2}{2t}-y\lambda + \frac{xy}{t}}dy \\ 
     \nonumber & \geq \frac{1}{\sqrt{2\pi t}}e^{(1-\frac{\lambda^2}{2})t}e^{\lambda x}e^{-\frac{x^2}{2t}} \int_{-K}^0 e^{-\frac{y^2}{2t}-y\lambda}dy \\
    \nonumber&\gtrsim  \frac{1}{\sqrt{t}}e^{(1-\frac{\lambda^2}{2})t}e^{\lambda x}e^{-\frac{x^2}{2t}}.
\end{align}

This implies that 
\begin{equation}\label{Z_t^1-goes-to-zero-for-tightness}
    \int_{-\infty}^{-bt} \frac{1}{\sqrt{t}}e^{(1-\frac{\lambda^2}{2})t}e^{\lambda x}e^{-\frac{x^2}{2t}} \theta(dx)\overset{\mathbb{P}}{\to} 0.
\end{equation}
Since $\theta \in \mathcal{N}_2$ almost surely, clearly for all $t_0 >1$, 
\[ \sup_{1\leq t \leq t_0} \int_{-\infty}^{-bt} \frac{1}{\sqrt{t}}e^{(1-\frac{\lambda^2}{2})t}e^{\lambda x}e^{-\frac{x^2}{2t}} \theta(dx) <\infty \textrm{ almost surely.} \]
Hence, the convergence \eqref{Z_t^1-goes-to-zero-for-tightness} implies the tightness of $\int_{-\infty}^{-bt} \frac{1}{\sqrt{t}}e^{(1-\frac{\lambda^2}{2})t}e^{\lambda x}e^{-\frac{x^2}{2t}} \theta(dx)$ in $t\geq 1$.

For the tightness of $\int_{-bt}^{0} \frac{1}{\sqrt{t}}e^{(1-\frac{\lambda^2}{2})t}e^{\lambda x}e^{-\frac{x^2}{2t}} \theta(dx)$, we note using \eqref{F-KPP-Formula-For-M}, 
\[ \mathbb{P}(|x + M_t -\lambda t| \leq K) = u_M(t, z + \sqrt{2}t - K) -u_M(t, z + \sqrt{2}t + K), \]
where $z = (\lambda - \sqrt{2})t - x$. For $x \in [-  bt, 0]$, $z \in [(\lambda -\sqrt{2})t, (\lambda -\sqrt{2} + b) t]$. Hence, using \eqref{quant-bound-for-super-critical-for-difference}, for $K$ large enough and $t\geq t_0$,  
\begin{equation} \label{replace-M-bound-by-its-lower-explicit-integral}
    \mathbb{P}(|x + M_t -\lambda t| \leq K) \gtrsim  \frac{1}{\sqrt{t}} e^{-\sqrt{2}z} e^{-\frac{z^2}{2t}}.
\end{equation}
Note that for $z = (\lambda - \sqrt{2})t - x$, 
\begin{equation}\label{z-to-x-transformation}
        \frac{1}{\sqrt{t}} e^{-\sqrt{2}z} e^{-\frac{z^2}{2t}} = \frac{1}{\sqrt{t}}e^{(1-\frac{\lambda^2}{2})t}e^{\lambda x}e^{-\frac{x^2}{2t}}.
    \end{equation}
Then, the tightness of $Z_t^b$ given by Proposition \ref{crucial-tightness-prop} completes the proof. 

\end{proof}

\subsection{Proof of convergence $I_t^4, I_t^5 \overset{\mathbb{P}}{\to} 0$.} \label{subsection-for-I4/I5}

Using Lemma \ref{when-does-u-go-to-zero}, we know that $u_\phi(t, \lambda t -x) \to 0$ uniformly over $x \leq 0$. Hence,  
\begin{equation}\label{we-can-remove-log}
     \int_{[-bt, -at]}\log(1- u_\phi(t, \lambda t -x))\theta(dx) = -(1+ o_t(1)) \int_{[-bt, -at]}u_\phi(t, \lambda t -x)\theta(dx).
\end{equation}
Also, by Lemma \ref{lemma_upper and lower for F-KPP}-$(a)$ and relation \eqref{z-to-x-transformation}, we have that for $x\in [-bt, -at]$, 
\begin{equation}
    u_\phi(t, \lambda t -x) \lesssim \frac{1}{\sqrt{t}}e^{(1-\frac{\lambda^2}{2})t}e^{\lambda x}e^{-\frac{x^2}{2t}}.
\end{equation}
Therefore, it suffices to prove that for $i=4,5,$  as $t\to \infty$ and then $L\to \infty$,  

\begin{equation}\label{resolved-step-for-4-and-5}
    \int_{E_t^i} \frac{1}{\sqrt{t}}e^{(1-\frac{\lambda^2}{2})t}e^{\lambda x}e^{-\frac{x^2}{2t}} \theta(dx) \overset{\mathbb{P}}{\to} 0.
\end{equation}

\vspace{3mm}
To this end, choose $u = t + \frac{L\sqrt{t}}{\sqrt{\lambda^2 -2}}$ for $i=4$ and $u = t -\frac{L\sqrt{t}}{\sqrt{\lambda^2 -2}}$ for $i=5$. For this choice of $u$, note that  $|x| =-x \geq \gamma t + L\sqrt{t} = \gamma u$ for $x\in E_t^4$ and $|x| =-x \leq \gamma t - L\sqrt{t} = \gamma u$ for $x\in E_t^5$. Hence, uniformly over $x\in E_t^4$ or $x\in E_t^5$, 

\begin{equation}
    \begin{aligned}
        \frac{\frac{1}{\sqrt{t}}e^{(1-\frac{\lambda^2}{2})t}e^{\lambda x}e^{-\frac{x^2}{2t}}}{\frac{1}{\sqrt{u}}e^{(1-\frac{\lambda^2}{2})u}e^{\lambda x}e^{-\frac{x^2}{2u}}} 
        &= \sqrt{\frac{u}{t}}e^{(\frac{\lambda^2}{2}-1)(u-t) + \frac{x^2}{2}(\frac{1}{u}-\frac{1}{t})}\\
        &\lesssim e^{(\frac{\lambda^2}{2}-1)(u -t) + \frac{(\lambda^2-2)u^2}{2}\frac{t-u}{tu}}\\
        &=e^{-\frac{L^2}{2}}.
    \end{aligned}
    \end{equation}

Therefore, for $i=4,5$, recalling $Z_t$ as defined in Proposition \ref{replace-M-by-explicit-expression-in-tightness}, we have  
\begin{equation}\label{compare for I_i}
    \int_{E_t^i}\frac{1}{\sqrt{t}}e^{(1-\frac{\lambda^2}{2})t}e^{\lambda x}e^{-\frac{x^2}{2t}}\theta(dx) \lesssim e^{-\frac{L^2}{2}} \int_{-\infty}^0 \frac{1}{\sqrt{u}}e^{(1-\frac{\lambda^2}{2})u}e^{\lambda x}e^{-\frac{x^2}{2u}} \theta(dx) = e^{-\frac{L^2}{2}} Z_u. 
\end{equation}
Since $Z_u$ is tight by Proposition \ref{replace-M-by-explicit-expression-in-tightness}, by taking $L\to \infty$, this implies \eqref{resolved-step-for-4-and-5} which completes the proof.

\subsection{Convergence of $I_t^6$.}

Since $I_t^i \overset{\mathbb{P}}{\to} 0$ for $1\leq i\leq 5$, it follows using  \eqref{Laplace-comp-using-F-KPP-3} and \eqref{we-can-remove-log} that for a fixed point $\Pi$, $\theta \in Dom_{\lambda}(\Pi)$, and $f\in C_c^+(\mathbb{R})$, 

\begin{equation}\label{resolving-to-set-E-6}
\mathbb{E}\biggl[\exp\biggl\{-\int_{E_t^6} u_\phi(t, \lambda t -x)\theta(dx)\biggr\}\biggr] \to \mathbb{E}[e^{-\langle f,\Pi\rangle}]
\end{equation}
as $t\to \infty$ and then $L\to \infty$. 

Now, define 
\begin{equation}\label{def-Z-t-L}
    Z_{t,L} := \int_{E_t^6} \frac{1}{\sqrt{t}}e^{(1-\frac{\lambda^2}{2})t}e^{\lambda x}e^{-\frac{x^2}{2t}} \theta(dx).
\end{equation}
For each fixed $L>0$, $E_t^6 \subset (-\infty,0)$ for all sufficiently large $t$, and hence $Z_{t,L} \le Z_{t}$. Therefore, using Proposition \ref{replace-M-by-explicit-expression-in-tightness}, $Z_{t,L}$ is tight in the regime under consideration. Whenever \(L_n\to\infty\), we may choose \(t_n\to\infty\) so rapidly that
\[
  L_n=o(\sqrt{t_n});                                                \tag{2}
\]
for instance, it is enough to require \(t_n\geq L_n^4\). Along a further subsequence $\tau=(t_n,L_n)$,
\begin{equation}\label{Z-tau-def}
    Z_{t_n,L_n} \overset{d}{\to} Z^{\tau}
\end{equation}
for some finite nonnegative random variable \(Z^\tau\).

Note that for $x\in E_t^6$, $z = (\lambda -\sqrt{2}) t -x = \alpha(\lambda) t + o_t(1)$, with 
\[ \alpha(\lambda) = \lambda + \sqrt{\lambda^2 -2} - \sqrt{2} = \mu -\sqrt{2}.\] 
Hence,  using Proposition \ref{C(a) definition} and \eqref{z-to-x-transformation}, for $f$ of the type \eqref{Form of f}, we have 

\begin{equation}
    u_\phi(t, \lambda t -x) \sim C(f, \alpha(\lambda)) \frac{1}{\sqrt{t}}e^{(1-\frac{\lambda^2}{2})t}e^{\lambda x}e^{-\frac{x^2}{2t}}
\end{equation}
uniformly over $x\in E_t^6$. Since $Z_{t_n, L_n} \overset{d}{\to } Z^\tau$, it follows that for all $f$ of type \eqref{Form of f},
\begin{equation}\label{u-integral-converges-to-Z-tau}
    \int_{E_t^6} u_\phi(t, \lambda t-x) \theta(dx) \overset{d}{\to} Z^{\tau} C(f, \alpha(\lambda))
\end{equation} 
along the subsequence $\tau$. 

Now, similarly to \eqref{theta_t-def}, let us define 
\begin{equation} \label{theta-t-L-def}
        \theta_{t,L}^\lambda:= \sum\limits_{i \in I, x_i \in E_t^6} \sum\limits_{k=1}^{n^i(t)} \delta_{x_i+\chi_k^i(t)-\lambda t}. 
    \end{equation}
    
Note that since $\theta_{t,L}^\lambda$ is starting from $\theta|_{E_t^6}$, which has a finite top particle, \textcolor{black}{and $\theta \in \mathcal N_2$}, Lemma \ref{Proposition M_2} gives that $\theta_{t,L}^\lambda$ also has a finite top particle almost surely. Hence, $\langle f,\theta_{t,L}^\lambda \rangle <\infty$ a.s. for $f$ of type \eqref{Form of f}. Then, it follows using \eqref{Laplace-comp-using-F-KPP}, \eqref{convegence-zero-for-x>-log(t)},\eqref{we-can-remove-log}, and \eqref{u-integral-converges-to-Z-tau} that for $f$ of the type \eqref{Form of f}, 
\begin{equation}
    \mathbb{E}[e^{-\langle f, \theta_{t,L}^\lambda \rangle}] \to  \mathbb{E}[e^{-Z^{\tau} C(f, \alpha(\lambda))}]
\end{equation}
along subsequence $\tau$. By \eqref{supercritical-E-Laplace-fromula-with-shift}, $\mathbb{E}[e^{-Z^{\tau} C(f, \alpha(\lambda))}]$ is the Laplace transform of $\langle f, \widetilde{\mathcal E}_\infty^\lambda(\cdot -S^\tau) \rangle$ where $S^\tau := \log (Z^\tau)/\mu$ is a $[-\infty, \infty)$-valued random variable ($Z^\tau$ can be zero with positive probability). It hence follows that 
\begin{equation}\label{convergence-conclusion}
    \langle f, \theta_{t,L}^\lambda \rangle \overset{d}{\to} \langle f, \widetilde{\mathcal E}_\infty^\lambda(\cdot -S^\tau) \rangle
\end{equation} 
for all $f$ of type \eqref{Form of f}.

Observe that functions of the form $1_{[c, d)}(x)$ for $-\infty< c<d <\infty$ can be written as difference of functions of type \eqref{Form of f}. Therefore, \eqref{convergence-conclusion} holds for functions $f$ in the span of $\{1_{[c, d)}(x) | -\infty< c<d <\infty\}$. It then follows from standard arguments that \eqref{convergence-conclusion} holds for $f\in C_c^+(\mathbb{R})$ as well. This in turn implies using \eqref{Laplace-comp-using-F-KPP-3} and \eqref{we-can-remove-log} again that for $f\in C_c^+(\mathbb{R})$,

\begin{equation}\label{convergence-of-Laplace-to-the-limit-we-want}
\mathbb{E}\biggl[\exp\biggl\{-\int_{E_t^6} u_\phi(t, \lambda t -x)\theta(dx)\biggr\}\biggr] \to \mathbb{E}[e^{-\langle f, \widetilde{\mathcal E}_\infty^\lambda(\cdot -S^\tau)\rangle}]
\end{equation}
along the subsequence $\tau$. Then, by comparing this to \eqref{resolving-to-set-E-6}, we conclude that 
\begin{equation}\label{equality-of-Pi-to-E}
    \Pi \overset{d}{=} \widetilde{\mathcal E}_\infty^\lambda(\cdot -S^\tau).
\end{equation}

Since the law of $\Pi$ does not depend on the subsequence $\tau$, it follows that $S^{\tau} \overset{d}{=} S$ for some $[-\infty, \infty)$ valued random variable $S$ for all subsequences $\tau$. As a result, we get
\begin{equation}\label{equality-of-Pi-to-E-without-tau}
    \Pi \overset{d}{=} \widetilde{\mathcal E}_\infty^\lambda(\cdot -S).
\end{equation}

Furthermore, this implies that all subsequential limits $Z^\tau$ given by \eqref{Z-tau-def} have the same law as $Z:= e^{\mu S}$, which in turn implies that for $Z_{t,L}$ given by \eqref{def-Z-t-L}, 
\begin{equation}\label{Z-t-converges-to-Z}
    Z_{t, L} \overset{d}{\to} Z
\end{equation}
as $t\to \infty$ and then $L\to \infty$.

\begin{proof}[Proof of Theorem \ref{main-thm} and Theorem \ref{main-thm-2} for $\lambda >\sqrt{2}$.]
    Theorem \ref{main-thm} has already been proved by by \eqref{equality-of-Pi-to-E-without-tau}. Also, for $\theta \in Dom_{\lambda}(\Pi)$, it satisfies (H1), (H2), (H3) by Proposition \ref{lacunary-prop}, Proposition \ref{crucial-tightness-prop}, Proposition \ref{replace-M-by-explicit-expression-in-tightness}, and \eqref{Z-t-converges-to-Z}. 
    
    Conversely, if $\theta$ satisfies (H1), (H2), (H3), it can be seen from the proofs above that $I_t^i \overset{\mathbb{P}}{\to} 0$ for $1\leq i\leq 5$ still holds. Furthermore, since (H3) gives \eqref{Z-t-converges-to-Z}, it follows as above that $\theta_{t,L}^\lambda \overset{d}{\to} \widetilde{\mathcal E}_\infty^\lambda(\cdot -S)$ with $S= \log Z /\mu$. Hence, $\theta \in Dom_{\lambda}(\Pi)$ with $\Pi \overset{d}{=} \widetilde{\mathcal E}_\infty^\lambda(\cdot -S)$.
\end{proof}

\section{Proof of Theorem \ref{main-thm} and Theorem \ref{main-thm-2} for $\lambda =\sqrt{2}$.} \label{Section proof critical}

The proof of Theorem \ref{main-thm} and Theorem \ref{main-thm-2} for $\lambda = \sqrt{2}$ again follows a strategy similar to that in the previous section. Much of the required groundwork was established in
\cite{Shekhar_Domain}; Proposition \ref{excerpt-from-CGS-AAP} therefore
allows us to shorten the critical proof substantially. 
\textcolor{black}{Compared with the case $\lambda>\sqrt{2}$, the critical regime requires different estimates for the F-KPP equation and relies on the exponentially-lacunary property in part (2) of Proposition \ref{lacunary-prop} to rule out the contribution from the intermediate regions.}\\ 

Let $\Pi$ be a fixed point of BBM with drift $\lambda = \sqrt{2}$ and $\theta \in Dom_{\lambda}(\Pi)$. Then, for $f\in C_c^+(\mathbb{R})$, Lemma \ref{Formula-for-Laplace-of-theta_t} gives 

\begin{equation}\label{Laplace-comp-using-F-KPP-4-critical}
\mathbb{E}[e^{-\langle f,\theta_t^{\sqrt{2}}\rangle}] = \mathbb{E}\biggl[\exp\biggl\{\int_{-\infty}^\infty\log(1- u_\phi(t, \sqrt{2} t -x))\theta(dx)\biggr\}\biggr],
\end{equation}
where $\phi(x) = 1-e^{-f(-x)}$. We assume $f$ is supported in $[-K, K]$.  

For constants $a>0$ small enough and $b>2\sqrt{2}$, we divide $\mathbb{R}$ into sets 
 
 \begin{align}
    \nonumber F_t^1 &= (bt, \infty), \\ \label{Def-sets-F} 
    F_t^2 &= [a \log t , bt] , \\ \nonumber
    F_t^3 &= [0, a\log t ), \\ \nonumber
    F_t^4 &= (-\infty, 0),  
\end{align}

and write 
\begin{equation}
    \int_{-\infty}^\infty\log(1- u_\phi(t, \sqrt{2} t -x))\theta(dx) = \sum_{i=1}^4 J_t^i,
\end{equation}
where
\begin{equation}
   J_t^i :=  \int_{F_t^i}\log(1- u_\phi(t, \sqrt{2} t -x))\theta(dx).
\end{equation}

We show that in the limit $t \to \infty$,  $J_t^1, J_t^2, J_t^3 \overset{\mathbb{P}}{\to} 0 $, and $J_t^4$ converges in distribution to a random variable. \\
\subsection{The proof of $J_t^1 \overset{\mathbb{P}}{\to} 0$}

We note, using \eqref{upper-bound-for-F-KPP-using-many-to-one} and Lemma \ref{Lemma Ratio 1}, that $u_\phi(t, \sqrt{2}t -x) \to 0 $ uniformly over $x \in (-\infty, -bt) \cup (bt, \infty)$ for any $b> 2 \sqrt{2}$. Hence, 

\begin{equation}
\int_{|x| >bt}\log(1- u_\phi(t, \sqrt{2} t -x))\theta(dx) = -(1+ o_t(1)) \int_{|x| >bt}u_\phi(t, \sqrt{2} t -x)\theta(dx).
\end{equation}
Also, choosing $u>0$ as in Lemma \ref{Lemma Ratio 1},
\begin{align}\label{reuse-of-B-tightness}
    &\int_{|x| > bt}e^t \mathbb{P}(| x+ B_t - \sqrt{2} t| \leq K) \theta(dx)\\ \nonumber& \leq \sup_{|x| >bt} \frac{e^t \mathbb{P}(|x+ B_t - \sqrt{2} t| \leq K)}{\mathbb{P}(|x+ B_u - \sqrt{2} u| \leq K)} \int_{|x| >bt}\mathbb{P}(|x+ B_u - \sqrt{2} u| \leq K) \theta(dx).
\end{align}

Then, Lemma \ref{Lemma Ratio 1} and Proposition \ref{crucial-tightness-prop} imply that 
\begin{equation}\label{larger-than-bt-does-not-contribute}
     \int_{|x| > bt} e^t \mathbb{P}(| x+ B_t - \sqrt{2} t| \leq K) \theta(dx) \overset{\mathbb{P}}{\to} 0,
\end{equation}
which in turn, using \eqref{upper-bound-for-F-KPP-using-many-to-one}, implies 
\begin{equation}
    \int_{|x| > bt} u_\phi(t, \sqrt{2}t -x) \theta(dx) \overset{\mathbb{P}}{\to} 0.
\end{equation}
This in particular implies $J_t^1 \overset{\mathbb{P}}{\to} 0$.

\subsection{The proof of $J_t^2 \overset{\mathbb{P}}{\to} 0$} The proof is immediate from Proposition \ref{lacunary-prop} as follows. Since $\theta$ is an integer-valued measure and $J_t^2\le0$, 

\begin{align}
    \nonumber \mathbb P(J_t^2\neq0)=\mathbb{P}(J_t^2<0) &= \mathbb{P}\biggl( \int_{[a \log t,bt]}\log(1- u_\phi(t, \sqrt{2} t -x))\theta(dx) <0\biggr) \\
    \nonumber &\leq \mathbb{P}(\theta([a \log t, bt]) >0) \\
    \nonumber &= \mathbb{P}(\theta([a \log t, bt]) \geq 1).
\end{align}

Then, since $\theta$ is exponentially-lacunary, using Remark \ref{about-lacunary}, $\mathbb{P}(\theta([a \log t, bt]) \geq 1) \to 0$ for all $b>a >0$. Thus \(J_t^2\to0\) in probability.\\

\subsection{A reformulation of the tightness property of $\theta \in Dom_{\lambda}(\Pi)$ for $\lambda = \sqrt{2}$.}

We note the corresponding variant of Proposition \ref{replace-M-by-explicit-expression-in-tightness} for $\lambda = \sqrt{2}$. 

\begin{proposition}\label{replace-M-by-explicit-expression-in-tightness-for-critical}
Let $\Pi$ be a fixed point of BBM with drift $\lambda =\sqrt{2}$ and $\theta \in Dom_{\lambda}(\Pi)$. Then, the random variables $\{\widetilde{Y}_t^b\}_{t\ge 1}$ given by 

\begin{equation}\label{tightness-tilde-Y-for-critical}
\widetilde{Y}_t^b := \int_{-\infty}^{\frac{1}{4} \log t }\frac{1}{t^{3/2}}(-x + \log t)e^{\sqrt{2}x} e^{-\frac{x^2}{2t}} \theta(dx)     
\end{equation}
is tight. 
\end{proposition} 

\begin{proof}
    The tightness of 
    \[\frac{1}{t^{3/2}}\int_{-bt}^{\frac{1}{4} \log t }(-x + \log t)e^{\sqrt{2}x} e^{-\frac{x^2}{2t}} \theta(dx)\]
 is an immediate consequence of part-(3) of Proposition \ref{crucial-tightness-prop} and lower bound in Lemma \ref{Lemma-on-difference-U_M-for-critical}. On the other hand, using \eqref{larger-than-bt-does-not-contribute}, we know that 
\begin{equation}
    \int_{-\infty}^{-bt} e^t \mathbb{P}(| x+ B_t - \sqrt{2} t| \leq K) \theta(dx) \overset{\mathbb{P}}{\to} 0,
\end{equation} 
and, similarly to \eqref{Brownian-in-negative-part-actually-gives-sharp-bound-super-critical}, for $x \leq -bt$, 
\begin{align} \label{Brownian-in-negative-part-actually-gives-sharp-bound-for-critical}
    \nonumber e^t \mathbb{P}(|x + B_t - \sqrt{2} t| \leq K) &=  \frac{1}{\sqrt{2\pi t}}e^{\sqrt{2} x}e^{-\frac{x^2}{2t}} \int_{-K}^K e^{-\frac{y^2}{2t}-\sqrt{2} y+ \frac{xy}{t}}dy \\ 
     \nonumber & \gtrsim \frac{1}{\sqrt{t}}e^{\sqrt{2} x}e^{-\frac{x^2}{2t}} \int_{-K}^0 e^{\frac{xy}{t}}dy \\
    \nonumber& =  \frac{1}{\sqrt{t}}e^{\sqrt{2} x}e^{-\frac{x^2}{2t}} \frac{t}{x}(1-e^{-Kx/t}) \\ 
    \nonumber & \gtrsim \frac{1}{\sqrt{t}}e^{\sqrt{2} x}e^{-\frac{x^2}{2t}} \frac{t}{-x}\frac{x^2}{t^2} \gtrsim \frac{1}{t^{3/2}}(-x + \log t)e^{\sqrt{2} x}e^{-\frac{x^2}{2t}},
\end{align}
where in last line we have used $e^r - 1 \gtrsim r^2$ for $ r \ge Kb$, with $r = -\frac{Kx}{t}$ for all sufficiently large $t$. It hence follows that 
\[\frac{1}{t^{3/2}}\int_{-\infty}^{-bt}(-x + \log t)e^{\sqrt{2}x} e^{-\frac{x^2}{2t}} \theta(dx) \overset{\mathbb{P}}{\to} 0.\]
This, together with the fact that $\theta \in \mathcal{N}_2$ almost surely, implies the tightness of 
\[\frac{1}{t^{3/2}}\int_{-\infty}^{-bt}(-x + \log t)e^{\sqrt{2}x} e^{-\frac{x^2}{2t}} \theta(dx),\]
which completes the proof. 
\end{proof}

\subsection{The proof of $J_t^3 \overset{\mathbb{P}}{\to} 0$}

Using Lemma \ref{when-does-u-go-to-zero}, for $a>0$ small enough, $u_\phi(t, \sqrt{2}t -x) \to 0 $ uniformly over $x \leq a \log t$, and 
\begin{equation}
\int_{F_t^3}\log(1- u_\phi(t, \sqrt{2} t -x))\theta(dx) = -(1+ o_t(1)) \int_{F_t^3}u_\phi(t, \sqrt{2} t -x)\theta(dx).
\end{equation}
Using \eqref{u-phi-bound-using-u-M} and \eqref{upper-bound-for-critical-U-M-for-x>-log t}, for $x \in [0, a \log t]$, 

\[ u_\phi(t, \sqrt{2}t -x) \leq u_M(t, \sqrt{2}t - x -K) \lesssim \frac{\log t} {t^{3/2}} e^{\sqrt{2}x}. \]
Hence, it suffices to show that 
\begin{equation} \label{sufficient-for-J-3}
\int_{0}^{a\log t} \frac{\log t} {t^{3/2}} e^{\sqrt{2}x} \theta(dx) \overset{\mathbb{P}}{\to} 0.
\end{equation}

To this end, note that for $x \in [0, \frac{1}{4}\log t ]$,
\[ \frac{\log t} {t^{3/2}} e^{\sqrt{2}x} \lesssim \frac{1}{t^{3/2}}(-x + \log t)e^{\sqrt{2}x} e^{-\frac{x^2}{2t}}.\]
Thus, Proposition \ref{replace-M-by-explicit-expression-in-tightness-for-critical}
implies that 
\begin{equation}\label{tightness-simplified-for-F-3}
    \int_{0}^{\frac{1}{4}\log t } \frac{\log t} {t^{3/2}} e^{\sqrt{2}x} \theta(dx)
\end{equation}
is tight in $t\geq 1$. Then, by taking $u=\sqrt{t}$, for $a$ small enough, 
\begin{equation}
    \int_{0}^{a \log t} \frac{\log t} {t^{3/2}} e^{\sqrt{2}x} \theta(dx) \lesssim \frac{1}{t^{3/4}}\int_{0}^{a \log t} \frac{\log u} {u^{3/2}} e^{\sqrt{2}x} \theta(dx) \lesssim \frac{1}{t^{3/4}}\int_{0}^{\frac{1}{4}\log u} \frac{\log u} {u^{3/2}} e^{\sqrt{2}x} \theta(dx).
\end{equation}

Therefore, the tightness in \eqref{tightness-simplified-for-F-3} implies \eqref{sufficient-for-J-3}, which completes the proof. 

\subsection{The convergence of $J_t^4$.}

Proposition \ref{replace-M-by-explicit-expression-in-tightness-for-critical} implies that 
\[ Y_t = \frac{1}{t^{3/2}}\int_{-\infty}^0(-x)e^{\sqrt{2}x} e^{-\frac{x^2}{2t}} \theta(dx) \]
is tight in $t\geq 1$. Let $\tau = \{t_n\}_{n\geq 1}$ be a sequence such that $Y_t \overset{d}{\to} Z^{\tau}$ along $\tau$ to some non-negative random variable $Z^\tau$. 

Similarly as \eqref{theta_t-def}, let us define 
\begin{equation} \label{theta_t-def-for-negative}
        \theta_{t}^{<}:= \sum\limits_{x_i <0} \sum\limits_{k=1}^{n^i(t)} \delta_{x_i+\chi_k^i(t)-\sqrt{2} t}.
    \end{equation}
Clearly, as in \eqref{Laplace-comp-using-F-KPP-4-critical}, we have 
\begin{equation}\label{laplace-for-negative-part}
    \mathbb{E}[e^{-\langle f,\theta_t^{<}\rangle}] = \mathbb{E}[e^{J_t^4}].
\end{equation}

Now, using Proposition \ref{excerpt-from-CGS-AAP}, along the sequence $\tau$, 
\begin{equation}\label{reuse-of-main-result-CGS-along-subsequence}
     \langle f, \theta_{t}^{<} \rangle \overset{d}{\to} \langle f, \widetilde{\mathcal{E}}_\infty(\cdot - S^\tau)\rangle
\end{equation}
for all $f \in C_c^+(\mathbb{R})$, where $S^\tau= \log Z^\tau/\sqrt{2}$. Hence, \textcolor{black}{along $\tau$},
\[ \mathbb{E}[e^{J_t^4}] \to \mathbb{E}[e^{-\langle f, \widetilde{\mathcal{E}}_\infty(\cdot - S^\tau)\rangle}].\]
But, since $\theta \in Dom_{\lambda}(\Pi)$ and $J_t^i\overset{\mathbb{P}}{\to} 0$ for $i=1,2,3$, \textcolor{black}{along $\tau$},
\[ \mathbb{E}[e^{J_t^4}] \to \mathbb{E}[e^{-\langle f, \Pi\rangle}].\]
This implies that 
\begin{equation}
    \Pi \overset{d}{=} \widetilde{\mathcal E}_\infty(\cdot -S^\tau).
\end{equation}

Since the law of $\Pi$ does not depend on the sequence $\tau$, it follows that $S^{\tau} \overset{d}{=} S$ for some $[-\infty, \infty)$-valued random variable $S$ for all sequences $\tau$, and

\begin{equation}\label{main-thm-concluded-for-critical}
    \Pi \overset{d}{=} \widetilde{\mathcal E}_\infty(\cdot -S).
\end{equation}

Furthermore, this implies that all subsequential limits $Z^\tau$ of $Y_t$ have the same law as $Z:= e^{\sqrt{2} S}$, which implies that 
\begin{equation}\label{Y-t-converges-to-Z}
    Y_t \overset{d}{\to} Z \qquad \text{as $t \to \infty$.}
\end{equation}

\begin{proof}[Proof of Theorem \ref{main-thm} and Theorem \ref{main-thm-2} for $\lambda =\sqrt{2}$.]

    Theorem \ref{main-thm} has already been proved by by \eqref{main-thm-concluded-for-critical}. Also, for $\theta \in Dom_{\lambda}(\Pi)$, it satisfies (A1), (A2), (A3) by Proposition \ref{lacunary-prop}, Proposition \ref{crucial-tightness-prop}, Proposition \ref{replace-M-by-explicit-expression-in-tightness-for-critical}, and \eqref{Y-t-converges-to-Z}. 
    
    Conversely, if $\theta$ satisfies (A1), (A2), (A3), it can be seen from the proofs above that $J_t^i \overset{\mathbb{P}}{\to} 0$ for $1\leq i\leq 3$ still holds. Furthermore, since (A3) gives \eqref{Y-t-converges-to-Z}, it follows from Proposition \ref{excerpt-from-CGS-AAP} that $\theta_{t}^{<} \overset{d}{\to} \widetilde{\mathcal E}_\infty(\cdot -S)$ with $S= \log Z /\sqrt{2}$. Hence $\theta \in Dom_{\lambda}(\Pi)$ with $\Pi \overset{d}{=} \widetilde{\mathcal E}_\infty(\cdot -S)$.
\end{proof}

\section{Some Additional Remarks}\label{Section Additional remarks}

\subsection{A remark on the finite intensity assumption in [Theorem 2.1-\cite{Kabluchko}].}\label{discussion-on-finite-intensity-vs-top-particle}

In the setting of fixed points of BBM, Theorem 2.1 of \cite{Kabluchko} states that if $\Pi$ is a fixed point of BBM with drift $\lambda >\sqrt{2}$ having locally finite intensity, then $\Pi \overset{d}{=} \widetilde{\mathcal{E}}_\infty^\lambda(\cdot-S)$ with a shift $S$ satisfying $\mathbb{E}[e^{\mu S}] <\infty$. To validate our main result Theorem \ref{main-thm}, we show in an a priori manner that the assumption of locally finite intensity forces $\Pi$ to have a finite top particle almost surely.

\begin{proposition}\label{Prop_intensity implies top particle}
Let $\Pi$ be a fixed point of BBM with drift $\lambda>\sqrt 2$ such that  

\begin{equation}\label{Pi-satisfies-on-compacts}
    \mathbb{E}[\Pi(K)] < \infty \textrm{ for all compact sets } K\subset \R. 
\end{equation}

Then,
\begin{equation}\label{Pi-satisfies-one-sided}
     \mathbb{E}[\Pi([0, \infty))] < \infty.
\end{equation}

In particular, $\Pi([0,\infty)) < \infty$ almost surely.  

\end{proposition}

\begin{proof}

Let $\theta_t^\lambda$ be given by \eqref{theta_t-def} started from $\theta_0^\lambda= \Pi$. Then, since $\Pi$ is a fixed point,  $\theta_t^\lambda(K) \overset{d}{=} \Pi(K)$ for all compact sets $K$ and for all $t>0$. Also, since $\Pi$ satisfies \eqref{Pi-satisfies-on-compacts}, $\E[\theta_t^\lambda(K) ] = \E[\Pi(K)]$. 
By conditioning on $\Pi$ and using the many-to-one lemma (Lemma \ref{Lemma many to one}), for $n\geq 1$, 

\begin{align} \label{Equation locally finite}
\nonumber\mathbb E[\theta_t^\lambda[n,n+1]\bigl|\Pi] &=\int_{\mathbb R} e^{t} \mathbb P(x+B_t-\lambda t \in [n,n+1])\Pi(dx) \\
&= \int_{\mathbb R} \int_{n}^{n+1} \frac{e^t}{\sqrt{2 \pi t}} \exp\{-(y+\lambda t-x)^2/2t\} dy \Pi(dx).
\end{align}

Let $\eta=\frac{1}{\lambda}$. Then, for sufficiently large $t$ such that $e^{\eta}\sqrt{t-\eta}/\sqrt{t} \geq \alpha >1$, we can write, \textcolor{black}{for all $n \ge 0$},

\begin{align*}
\mathbb E[\theta_t^\lambda[n,n+1]\bigl|\Pi]&= \int_{\mathbb R} \int_{n}^{n+1} \frac{e^t}{\sqrt{2 \pi t}}  \exp\biggl\{-\frac{(y+1+\lambda (t-\eta)-x)^2}{2t}\biggr\} dy \Pi(dx)\\
& \ge e^{\eta}\frac{\sqrt{t-\eta}}{\sqrt{t}}\int_{\mathbb R} \int_{n+1}^{n+2} \frac{e^{t-\eta}}{\sqrt{2\pi (t-\eta)}}  \exp\biggl\{-\frac{(y+\lambda (t-\eta)-x)^2}{2(t-\eta)}\biggr\} dy \Pi(dx)\\
&\ge \alpha \mathbb E[\theta_{t-\eta}^\lambda[n+1,n+2]\bigl|\Pi],
\end{align*}

where we used \eqref{Equation locally finite} once again in the last line. Taking expectation on both sides above, it follows that 
\[ \mathbb{E}[\Pi([n+1, n+2])] \leq \frac{1}{\alpha}  \mathbb{E}[\Pi([n, n+1])],\]

which by iteration implies 

\[\mathbb{E}[\Pi([n, n+1])] \lesssim \alpha^{-n} \mathbb{E}[\Pi([0, 1])].\]

Since $\alpha >1$, the right-hand side is summable, which implies \eqref{Pi-satisfies-one-sided}.

\end{proof}
 
\subsection{A discussion on Tauberian theorems.}

The condition (A3) appearing in Theorem \ref{main-thm-2} for $\lambda = \sqrt{2}$ was further simplified in \cite{Shekhar_Domain} using a probabilistic version of the Hardy-Littlewood-Karamata (HLK) Tauberian theorem, see [Section 5-\cite{Shekhar_Domain}] for details, and also see \cite{Berestycki-Wong-24-Weyl-Law} where such a probabilistic HLK Tauberian theorem appears as well in the context of Weyl's law in Liouville quantum gravity. It is then natural to ask if it is possible to simplify/convert the condition (H3) appearing in Theorem \ref{main-thm-2} for $\lambda >\sqrt{2}$ into a simpler statement about $\theta$. It can be easily shown using the same arguments as in section \ref{subsection-for-I4/I5} that in the presence of tightness of $Z_t := \int_{- \infty}^0 \frac{1}{\sqrt{t}}e^{(1-\frac{\lambda^2}{2})t}e^{\lambda x}e^{-\frac{x^2}{2t}}\theta(dx)$, (H3) can be equivalently stated as  

\begin{equation}\label{conclusion-for-supercritical-2}
   Z_t = \int_{-\infty}^{0} \frac{1}{\sqrt{t}}e^{(1-\frac{\lambda^2}{2})t}e^{\lambda x}e^{-\frac{x^2} {2t}}\theta(dx) \overset{d}{\to} Z.
    \end{equation}

By defining $\hat{\theta}(A) = \theta(-A)$, $G(x) = \int_0^{\sqrt{x}} e^{-\lambda y} \hat{\theta}(dy)$ and $\delta = 1/2t$, the condition \eqref{conclusion-for-supercritical-2} can be recast into the form  

\begin{equation}\label{exponential-tauberian-statement}
    \delta^{l} e^{b/\delta} \int_0^\infty e^{-\delta x} dG(x) \overset{d}{\to} X,
\end{equation}
as $\delta \to 0+$, where $G$ is a random monotonic increasing function, $l=1/2$, $b$ is some \textcolor{black}{negative} constant (\textcolor{black}{since $\lambda>\sqrt 2$}) and $X$ is some random variable. Then, simplifying condition (H3) is equivalent to converting \eqref{exponential-tauberian-statement} into an asymptotic statement about function $G(x)$ as $x\to + \infty$. This is a natural problem in Tauberian theory. In the critical case $\lambda = \sqrt{2}$ (which gives $b=0$ in \eqref{exponential-tauberian-statement}), such an attempt falls exactly in the framework of HLK Tauberian theorem. In this case, the decay of the Laplace transform of $G$ is of polynomial type rather than of exponential type in \eqref{exponential-tauberian-statement}. One can hence rely on the theory of regularly varying functions, which is the framework of HLK Tauberian theorem, see [Section 5-\cite{Shekhar_Domain}] for details. However, for $\lambda >\sqrt{2}$ (which gives $b<0$ in \eqref{exponential-tauberian-statement}), without assuming any additional assumption on the function $G$, \eqref{exponential-tauberian-statement} cannot be always converted into an asymptotic statement about $G(x)$ as $x\to \infty$. A counterexample is provided below. However, there are some related results in that direction, e.g. Kohlbecker's, Cadena's, Kosugi's, Kasahara's, de Bruijn's Tauberian theorems, see [Chapter 4.12-\cite{Bingham_Regular_variation}] and \cite{cadena2015note, Nobuko_Tauberian}. In particular, when $G$ is deterministic, Kohlbecker's Tauberian theorem can be used to conclude from \eqref{exponential-tauberian-statement} that $\log G(x) \sim 2 \sqrt{-bx}$.

\subsubsection{A counterexample to exponential type Tauberian equivalence.}\label{rem:tauberian}

We consider the special case \(b=-1\) and \(l=0\) in
\eqref{exponential-tauberian-statement}. The underlying construction
goes back to \cite{avakumovic1936einige}. Here we use the version given in \cite{bringmann2023tauberian}, where the authors  make some minor modifications  in order to obtain a power series instead of the
continuous Laplace transform used in the original example.

Let
\[
    q_m=m^{-1/4}e^{2m^{3/2}},\qquad
    \nu=\sum_{m=1}^\infty q_m\delta_{m^3},\qquad
    G_1(x)=\nu([0,x]).
\]
Then, as \(\delta\to0^+\), we have
\begin{equation}\label{exp-tau-statement-2}
    \int_0^\infty e^{-\delta x}\,dG_1(x)
    =\sum_{m\ge1}m^{-1/4}e^{2m^{3/2}-\delta m^3}
    \sim \frac{2\sqrt{\pi}}{3}e^{1/\delta}.
\end{equation}

For comparison, define
\[
    G_0(x)=\frac{1}{3}x^{-1/4}e^{2\sqrt{x}},
    \qquad x\ge1,
\]
and extend \(G_0\) monotonically to \([0,1]\). A Laplace-method calculation
gives the same asymptotic as in \eqref{exp-tau-statement-2}; the extension
on a bounded interval does not affect the leading term. Thus \(G_0\) is a
natural candidate for a first-order asymptotic equivalent, but
\eqref{exp-tau-statement-2} does not force \(G_1(x)\sim G_0(x)\).

Indeed, since \(q_{m-1}/q_m\to0\), we have \(G_1(m^3)\sim q_m\). Hence
\[
    \frac{G_1(m^3)}{G_0(m^3)}
    \sim 3\sqrt{m}\longrightarrow\infty.
\]
On the other hand, for \(y_m=(m+1)^3-1\), we again have
\(G_1(y_m)\sim q_m\), whereas
\[
    \frac{G_1(y_m)}{G_0(y_m)}
    \sim 3\sqrt{m}\,e^{-3\sqrt{m}}
    \longrightarrow 0.
\]
Consequently,
\[
    \limsup_{x\to\infty}\frac{G_1(x)}{G_0(x)}=+\infty,
    \qquad
    \liminf_{x\to\infty}\frac{G_1(x)}{G_0(x)}=0.
\]
Hence \(G_1\) is not asymptotic to \(G_0\). In this example,
\(\log G_1(x)\sim2\sqrt{x}\), in agreement with the logarithmic conclusion
above.

\section{Appendix} \label{Section Appendix}

For a Poisson point process $PPP(\eta(x)dx)$ and another independent point process $\mathcal{D}$, we will use the notation $DPPP(\eta(x)dx, \mathcal{D})$ to denote the decorated Poisson point process obtained from these.
For $\rho \geq \sqrt{2}$, let $\mathcal{E}^{\rho}= DPPP(Ce^{-\rho x}dx, \mathcal D^{\rho})$, where $C$ is a constant and $\mathcal{D}^\rho$ is as in \eqref{decoration_family}. A direct computation gives that
    \begin{equation}\label{Laplace_DPPP} 
    \mathbb{E}[e^{-\langle f,\mathcal{E}^{\rho}\rangle}] = \exp\left(-C\int_{-\infty}^{\infty} \mathbb E[1-e^{-\int f(y+z)\mathcal{D}^{\rho}(dz)}]e^{-\rho y}dy\right),
\end{equation}
for every $f\in C_{c}^{+}(\mathbb{R})$ or of the form \eqref{Form of f}.

\subsection{Identification of the fixed point (for $\lambda>\sqrt 2)$}\label{sec3}
Recall that $\widetilde{\mathcal E}_\infty^\lambda$ is the large time limit of the BBM flow with supercritical drift $\lambda>\sqrt 2$ starting from the initial distribution  $ PPP(\frac{1}{\sqrt{2\pi}}e^{-\mu x } dx)$, with $\mu=\lambda+\sqrt{\lambda^2-2}$ (see Section \ref{main-results-subsection}). 
Now we determine the exact structure of $\widetilde{\mathcal E}_\infty^\lambda$, and also motivate why this particular choice of $\mu$ is essential.
The approach is inspired by \cite{Lalley_Sellke} and \cite{Kabluchko}. 
Consider the BBM flow $\theta_t^\lambda$ with the initial distribution $\theta_0= PPP(\eta(x)dx)$ and we study the corresponding $\theta^\lambda_{t}$ as $t\to\infty$.
One important observation is that for some appropriately chosen $\eta(x)$, the intensity of $\theta^\lambda_t$ is time invariant, that is 
for all $f$ of the form \eqref{Form of f} \footnote{It suffices to consider only this class of test functions as a posteriori, the right choice of $\eta$ implies PPP$(\eta(x)dx)\in \mathcal N_2$ (see Section \ref{Section Laplace}).},
\begin{equation}\label{Eqn_equal expectation_app1}
  \mathbb{E} [\langle f,\theta_{t}^{\lambda}\rangle] = \mathbb{E} [\langle f, \theta\rangle ], \quad \forall t \ge 0.
\end{equation}
To find such a suitable intensity measure $\eta(x)dx$, setting $v_f(t) := \mathbb{E}[\langle f, \theta_{t}^{\lambda}\rangle ]$, \eqref{Eqn_equal expectation_app1} implies
\begin{equation}\label{3.3'}
  \frac{\partial v_f}{\partial t} = 0.
\end{equation}
Using the many to one formula (Lemma \ref{Lemma many to one}), we have 
\begin{equation}
  v_f(t) = \int_{\mathbb{R}}e^{t}\mathbb{E}[f(x+B_{t}-\lambda t)]\eta(x)dx.
\end{equation}
Moreover, $\omega(t,x) := \mathbb{E}[f(x+B_{t}-\lambda t)]$ satisfies
\begin{equation}
  \frac{\partial \omega}{\partial t} = \frac{1}{2} \frac{\partial^2\omega}{\partial x^2} -\lambda\frac{\partial \omega}{\partial x} .
\end{equation}
Combining these and using integration by parts, we get that $\eta(x)$ satisfies the following ODE:
 \begin{equation}
  \frac{1}{2}\eta'' + \lambda\eta' + \eta = 0.
 \end{equation}
 For $\lambda>\sqrt{2}$, the above ODE has the following two fundamental solutions:
\[ \eta_{1}(x) = e^{-\mu_{1}x}, \quad  \eta_{2}(x) = e^{-\mu_{2}x},\]
where $\mu_{1} = \lambda-\sqrt{\lambda^{2}-2}$, $\mu_{2} = \lambda+\sqrt{\lambda^{2}-2}$. 

It turns out that both choices of $\eta$ ensure that the intensity of $\theta^\lambda_t$ is time-invariant.
 In \cite{Lalley_Sellke}, the authors showed that if we start from $PPP(e^{-\mu_{1}x}dx)$, $\theta^\lambda_{t}$ converges in distribution to the empty point process, that is, the Laplace functional of $\theta^\lambda_{t}$ converges to $1$ as $t \to \infty$. 
 They also conjectured that starting from $PPP(e^{-\mu_{2}x}dx)$, $\theta^\lambda_{t}$ converges in distribution to a nontrivial limit.
 Kabluchko showed that this conjecture is true and the limiting point process is a decorated Poisson point process \cite{Kabluchko}. 
We now identify the exact structure of the limiting point process in the following proposition.
\begin{proposition}\label{Identify_DPPP}
  For $\lambda> \sqrt{2}$, let $\mu = \lambda+\sqrt{\lambda^{2}-2}$, $\theta_0= PPP(\frac{1}{\sqrt{2\pi}}e^{-\mu x}dx),$ then as $t \to \infty$,
\begin{equation}
  \theta_{t}^{\lambda} \Rightarrow \widetilde{\mathcal E}_{\infty}^\lambda,
\end{equation}
in distribution for the vague topology
on \(\mathcal N\). Furthermore 
\begin{equation}\label{Def_fixed}
 \widetilde{\mathcal E}_{\infty}^\lambda \overset{d}{=} DPPP(\mu C_M(\alpha(\lambda))e^{-\mu x}dx,\mathcal{D}^{\mu}),
\end{equation}
 where $C_M(\cdot)$ is as in \eqref{rename-C(f)-for-M-for-supercritical} and
$\mathcal D^{\mu}(\cdot)$ is a decoration point process with the law
\begin{equation}
     \mathbb P(\mathcal{D}^{\mu} \in \cdot) = \lim_{t\to\infty}\mathbb P(\sum_{k = 1}^{n(t)}\delta_{\{\chi_{k}(t) - M_{t} \}} \in \cdot | M_{t} \ge \mu t).
\end{equation}
\end{proposition}

The rest of this section is devoted to proving this proposition. 
Recall that $\{\chi^i\}_{i \in \mathbb N}$ is a family of i.i.d. BBMs starting from the origin, $n^i(t)$ counts the number of particles of the $i$-th BBM alive at time $t$, and $\{\chi_k^i(t)\}_{k \le n^i(t)}$ are their locations.
For $f$ as in \eqref{Form of f} and $supp (f) \subset [-K_f, \infty), $ for some $K_f>0$, consider the Laplace transform of $\theta_{t}^\lambda:$
\begin{align}
 \nonumber \mathbb{E}[e^{-\langle f,\theta_{t}^{\lambda}\rangle}]
&=\mathbb{E}[e^{-\sum_{i=1}^{\infty}\sum_{k=1 }^{n^i(t)}f(x_{i}+\chi_{k}^{i}(t)-\lambda t)}]\\
\nonumber &=\exp\left(-\frac{1}{\sqrt{2\pi}}\int (1-\mathbb E[e^{-\sum_{k=1 }^{n^i(t)}f(x_{i}+\chi_{k}^{i}(t)-\lambda t)}])e^{-\mu x}dx\right)\\
 \label{Laplace from PPP} &=\exp\left(-\frac{1}{\sqrt{2\pi}}\int_{-\infty}^{\infty}u_{\phi }(t,\lambda t-x)e^{-\mu x}dx\right) ,
\end{align}
where $u_{\phi}$ is the solution of the F-KPP equation
with initial data $u(0,x) = 1-e^{-f(-x)}$, and $f$ is of the form \eqref{Form of f}.
Thus, our objective is reduced to understanding the limiting behaviour of $\int_{-\infty}^{\infty}u_{\phi }(t,\lambda t-x)e^{-\mu x}dx$ as $t$ tends to $\infty$.

Towards this goal, we prove the next lemma, which states that for the initial distribution $ \theta_0= PPP(\frac{1}{\sqrt{2\pi}}e^{-\mu x}dx)$, only the BBMs starting from the region around $-(\mu-\lambda)t$ meaningfully contribute in large times.
For brevity, we prove this for the initial distribution $PPP(e^{-\mu x}dx)$, and it is straightforward to notice that the statement of the following lemma does not change if one considers the initial distribution $PPP(Ce^{-\mu x}dx)$, for any $C>0.$

\begin{lemma}\label{Main_contribution in sec3}
  Let $(p_{i})_{i\in I}$ be the atoms of a Poisson point process $\mathcal P$ with intensity measure $e^{-\mu x}dx$. For any choice of $z\in \mathbb R$ and $\epsilon>0$, there exist $t_{0}>0, \delta >0$ such that for all $t > t_{0}$, 
  \begin{equation}
    \mathbb{P} \biggl(\exists i\in \mathbb N, k \le n^{i}(t):p_{i}+\chi_k^i(t)-\lambda t \ge z,p_i\notin[-(\mu - \lambda)t-\delta\sqrt{t},-(\mu -\lambda)t+\delta\sqrt{t}]\biggr)\le \epsilon. \label{2.18}
  \end{equation}
\end{lemma}
\begin{proof}
  We first consider the case when $p_{i}> -(\mu -\lambda)t+\delta\sqrt{t}$.
  By union bounds and Campbell's formula, we have 
  \begin{equation}\label{estimae in Lem3.2}
    \begin{aligned}
       &\mathbb{P} \biggl(\exists i\in \mathbb N, k \le n^{i}(t):p_{i}+\chi_k^i(t)-\lambda t \ge z, p_i>-(\mu -\lambda)t+\delta\sqrt{t} \biggr)\\
     \le&\mathbb{E} \biggl[\sum_{p_{i} > -(\mu - \lambda)t +\delta \sqrt{t}}\mathbb{P} (p_{i}+M_{t}-\lambda t \ge z)\biggr]\\
      =&\int_{-(\mu-\lambda)t+\delta\sqrt{t}}^{\infty}\mathbb{P} (M_{t}\ge \lambda t-x+z)e^{-\mu x}dx.\\
    \end{aligned}
  \end{equation}
We use the many to one lemma  (Lemma \ref{Lemma many to one}) to bound the probability inside the integral and split it into two parts:
\begin{equation}
    \begin{aligned}
    &\int_{-(\mu-\lambda)t+\delta\sqrt{t}}^{\infty}\mathbb{P} (M_{t}\ge \lambda t-x+z)e^{-\mu x}dx\\
           \le&\int_{-(\mu-\lambda)t+\delta\sqrt{t}}^{\infty}e^{t}\mathbb{P}(B_{t}\ge \lambda t-x+z)e^{-\mu x}dx \\
    = &\int_{-(\mu-\lambda)t+\delta\sqrt{t}}^{0}e^{t}\mathbb{P}(B_{t}\ge \lambda t-x+z)e^{-\mu x}dx + \int_{0}^{\infty}e^{t}\mathbb{P}(B_{t}\ge \lambda t-x+z)e^{-\mu x}dx\\
     = &(I) + (II). \label{2.19}
    \end{aligned}
\end{equation}
By standard Gaussian Mills ratio argument, we get
  \begin{equation}
    \begin{aligned}
      (I)
     =&\int_{-(\mu-\lambda)t+\delta\sqrt{t}}^{0}e^{t}\mathbb{P}(B_{t}\ge \lambda t-x+z)e^{-\mu x}dx\\
\lesssim & \int_{0}^{(\mu-\lambda)t-\delta\sqrt{t}}\frac{e^{t}}{\sqrt{2\pi t}}e^{\mu x}e^{-\frac{(z+x+\lambda t)^{2}}{2t}}dx.\\
    \end{aligned}
  \end{equation}
  where we do a change of variables $x\to -x$. 
  Observe that since $\mu = \lambda + \sqrt{\lambda^2-2}$, we can find a constant $c_{1}>0$ depending on $z$ such that for  $t$ large enough,
  \begin{equation}
    \begin{aligned}
      &\int_{0}^{(\mu-\lambda)t-\delta\sqrt{t}}\frac{e^{t}}{\sqrt{2\pi t}}e^{\mu x}e^{-\frac{(z+x+\lambda t)^{2}}{2t}}dx\\
    \le& c_{1}\int_{0}^{(\mu-\lambda)t-\delta\sqrt{t}}\frac{1}{\sqrt{2\pi t}}e^{-\frac{(x-(\mu -\lambda)t)^{2}}{2t}}dx\\
     = & c_{1}\int_{-(\mu-\lambda)\sqrt{t}}^{-\delta}\frac{1}{\sqrt{2\pi}}e^{-\frac{u^{2}}{2}}du.\\
    \end{aligned}
  \end{equation}
  As for the second term, direct calculation gives that 
  \begin{equation}
    \begin{aligned}
      (II) 
      =&\int_{0}^{\infty}\frac{e^{t}}{\sqrt{2\pi t}}e^{-\mu x}dx\int_{0}^{\infty}e^{-\frac{[u+(z-x+\lambda t)]^{2}}{2t}}du\\
      =&e^{(1-\frac{\lambda^{2}}{2})t}\int_{0}^{\infty}e^{-(\mu-\lambda) x}dx\int_{0}^{\infty}\frac{1}{\sqrt{2\pi t}}e^{-\frac{[u+(z-x)]^{2}}{2t}}e^{-\lambda u-\lambda z}du.
    \end{aligned}
  \end{equation}
Note that $\int_{0}^{\infty}e^{-(\mu-\lambda) x}dx\int_{0}^{\infty}\frac{1}{\sqrt{2\pi t}}e^{-\frac{[u+(z-x)]^{2}}{2t}}e^{-\lambda u-\lambda z}du$ is finite as $\mu > \lambda$; therefore, one can find a constant $c_{2}$ depending on $z$ such that the above is bounded by $ c_{2}e^{(1-\frac{\lambda^{2}}{2})t}.$ 
By a similar computation used to prove the bound on $(I)$, we get
  \begin{equation}
    \mathbb{P}\biggl(\exists i\in I, k \le {n}^{i}(t):p_{i}+\chi^{i}_{k}(t)-\lambda t \ge z, p_{i}<-(\mu - \lambda)t-\delta\sqrt{t}\biggr) \le c_{3}\int_{\delta}^{\infty}e^{-\frac{u^{2}}{2}}du. \label{2.20}
  \end{equation}
  Therefore
     \begin{align*}
      &\mathbb{P} \biggl(\exists i\in \mathbb N, k \le n^{i}(t):p_{i}+\chi_k^i(t)-\lambda t \ge z,p_i\notin[-(\mu - \lambda)t-\delta\sqrt{t},-(\mu -\lambda)t+\delta\sqrt{t}]\biggr)\\
  \le&c_{1}\int_{-(\mu-\lambda)\sqrt{t}}^{-\delta}\frac{1}{\sqrt{2\pi}}e^{-\frac{u^{2}}{2}}du +c_{2}e^{(1-\frac{\lambda^{2}}{2})t}+ c_{3}\int_{\delta}^{\infty}e^{-\frac{u^{2}}{2}}du,
     \end{align*}
which concludes the proof of this lemma once we choose $\delta, t_0$ sufficiently large.
\end{proof}

\begin{proof}[Proof of Proposition \ref{Identify_DPPP}]
Lemma~\ref{Main_contribution in sec3} suggests that we should focus on the region $[-(\mu-\lambda)t - \delta\sqrt{t},-(\mu - \lambda)t + \delta\sqrt{t}]$. We next determine the limit of the BBM evolution when the Poisson initial process is restricted to
\(
 [-(\mu-\lambda)t-\delta\sqrt t,
   -(\mu-\lambda)t+\delta\sqrt t].
\)
For this purpose, we need a uniform description of a BBM conditioned to have
an unusually large maximum.
We use a similar approach as in [\cite{bovier2014extremal}-Corollary 7.6] to prove this.
\begin{lemma}\label{gap_convergence}
    For any fixed $\delta >0$ and $f$ of the form \ref{Form of f} with $supp (f) \subset [-K_f, \infty), $ for some $K_f>0$, uniformly in $x\in[-(\mu -\lambda)t-\delta\sqrt{t},-(\mu -\lambda)t+\delta\sqrt{t}]$, we have
\begin{equation}
  \begin{aligned}
    &\lim_{t\to\infty}\mathbb{E}[1-e^{-\sum_{k= 1}^{n(t)} f(x+\chi_{k}(t) - \lambda t)}|x+M_{t} -\lambda t\ge -K_{f}]\\
    & = \int_{-\infty}^{\infty}\mu \mathbb{E}[1-e^{-\int f(y+z)\mathcal{D}^{\mu}(dz) }]e^{-\mu y}e^{-\mu K_{f}}dy. \label{Eqn_contional M_t_app1}
 \end{aligned}
\end{equation}
\end{lemma}

We postpone the proof of this lemma until the end of this section. 
By Lemma \ref{C(a) definition}, for every fixed $\delta >0$, uniformly in $x\in[-(\mu-\lambda) t-\delta\sqrt{t},-(\mu -\lambda) t+\delta\sqrt{t}]$,
\begin{equation}
 \sqrt t\,e^{(\lambda^2/2-1)t}e^{-\lambda x+x^2/(2t)}
 u_M(t,\lambda t-x-K_f)
 \longrightarrow C_M(\alpha(\lambda))e^{\mu K_f}.  \label{Eqn_convergence C_M_app1}
\end{equation}
By Lemma \ref{gap_convergence} and \eqref{Eqn_convergence C_M_app1},
\begin{equation}
    \begin{aligned}
        &\int_{-(\mu -\lambda)t-\delta\sqrt{t}}^{-(\mu-\lambda)t+\delta\sqrt{t}}u_{\phi}(t,\lambda t-x)e^{-\mu x}dx\\
        =&\int_{-(\mu -\lambda)t-\delta\sqrt{t}}^{-(\mu-\lambda)t+\delta\sqrt{t}}\mathbb{E}[1-e^{-\sum_{k= 1}^{n(t)} f(x+\chi_{k}(t) - \lambda t)}|x+M_{t} -\lambda t\ge -K_{f}]\\
        & \times \mathbb{P}(x+M_{t} -\lambda t\ge -K_{f})e^{-\mu x}dx\\
        =&\int_{-(\mu -\lambda)t-\delta\sqrt{t}}^{-(\mu-\lambda)t+\delta\sqrt{t}}\mathbb{E}[1-e^{-\sum_{k = 1}^{n(t)} f(x+\chi_{k}(t) - \lambda t)}|x+M_{t} -\lambda t\ge -K_{f}]\\
        & \times u_{M}(t,\lambda t-x-K_{f})e^{-\mu x}dx.
    \end{aligned}
\end{equation}
First taking $t\to\infty$ and then $\delta\to\infty$, from \eqref{Eqn_contional M_t_app1} and \eqref{Eqn_convergence C_M_app1}, we get 
\begin{equation}
    \begin{aligned}
        &\lim_{\delta \to\infty}\lim_{t\to\infty}\frac{1}{\sqrt{2\pi}}\int_{-(\mu -\lambda)t-\delta\sqrt{t}}^{-(\mu-\lambda)t+\delta\sqrt{t}}u_{\phi}(t,\lambda t-x)e^{-\mu x}dx\\
        =&\mu C_{M}(\alpha(\lambda))\int_{-\infty}^{\infty}\mathbb{E}[1-e^{-\int f(y+z)\mathcal{D}^{\mu}(dz) }]e^{-\mu y}dy\\
    &\qquad \times \lim_{\delta\to \infty}\lim_{t\to\infty}\int_{-(\mu -\lambda)t-\delta\sqrt{t}}^{-(\mu-\lambda)t+\delta\sqrt{t}}\frac{1}{\sqrt{2\pi t}}e^{-\frac{[x+(\mu-\lambda)t]^{2}}{2t}}dx\\
    =&\mu C_{M}(\alpha(\lambda))\int_{-\infty}^{\infty}\mathbb{E}[1-e^{-\int f(y+z)\mathcal{D}^{\mu}(dz) }]e^{-\mu y}dy. 
    \end{aligned}
\end{equation}
Our last step is to show that first letting $t\to\infty$, and then $\delta\to\infty$,
\[\int_{x\in [-(\mu - \lambda)t-\delta\sqrt{t},-(\mu -\lambda)t+\delta\sqrt{t}]^c}u_{\phi}(t,\lambda t-x)e^{-\mu x}dx \to 0. \]
Observe that  
\begin{equation}
    \begin{aligned}
    &\int_{x>-(\mu-\lambda)t+\delta\sqrt{t}}u_{\phi}(t,\lambda t-x)e^{-\mu x}dx\\
    \le&\int_{x>-(\mu-\lambda)t+\delta\sqrt{t}}\mathbb{P}(x+ M_{t}>\lambda t -K_{f})e^{-\mu x}dx,
    \end{aligned}
\end{equation}
which is what we have estimated in \eqref{estimae in Lem3.2}. Therefore
\begin{equation}
    \lim_{\delta \to\infty}\lim_{t\to\infty}\int_{x>-(\mu-\lambda)t+\delta\sqrt{t}}u_{\phi}(t,\lambda t-x)e^{-\mu x}dx = 0.
\end{equation}
Similarly, we get
\begin{equation}
    \lim_{\delta \to\infty}\lim_{t\to\infty}\int_{x<-(\mu-\lambda)t-\delta\sqrt{t}}u_{\phi}(t,\lambda t-x)e^{-\mu x}dx = 0.
\end{equation}
Putting these together, we have 
\begin{equation}\label{Eqn_limit F-KPP integral}
  \begin{aligned}
    \lim_{t\to\infty} \frac{1}{\sqrt{2\pi}}\int_{-\infty}^{\infty}u_{\phi}(t,\lambda t-x)e^{-\mu x}dx
   =\mu C_{M}(\alpha(\lambda))\int_{-\infty}^{\infty}\mathbb{E}[1-e^{-\int f(y+z)\mathcal{D}^{\mu}(dz) }]e^{-\mu y}dy. 
  \end{aligned}
\end{equation}
By Lemma~\ref{gap_convergence}, uniformly in $x\in[-(\mu -\lambda)t-\delta\sqrt{t},-(\mu -\lambda)t+\delta\sqrt{t}]$, 
\begin{equation}
\begin{aligned}
     &\lim_{t\to\infty}\frac{\mathbb{E}[1-e^{-\sum_{k = 1}^{n(t)} f(x+\chi_{k}(t) - \lambda t)}]}{\mathbb{P}(x+M_{t} -\lambda t\ge -K_{f})}\\
     =&\lim_{t\to\infty}\mathbb{E}[1-e^{-\sum_{k= 1}^{n(t)} f(x+\chi_{k}(t) - \lambda t)}|x+M_{t} -\lambda t\ge -K_{f}]\\
      =&\int_{-\infty}^{\infty}\mu \mathbb{E}[1-e^{-\int f(y+z)\mathcal{D}^{\mu}(dz) }]e^{-\mu y}e^{-\mu K_{f}}dy.
\end{aligned}
\end{equation}
Rewriting this in terms of the solution of the F-KPP equation gives
\begin{equation}\label{L(f)_2}
    \lim_{t\to\infty}\frac{u_{\phi}(t,\lambda t-x)}{u_{M}(t,\lambda t-x-K_f)} = \int_{-\infty}^{\infty}\mu \mathbb{E}[1-e^{-\int f(y+z)\mathcal{D}^{\mu}(dz) }]e^{-\mu y}e^{-\mu K_{f}}dy.
\end{equation}
Applying Lemma \ref{C(a) definition} again, uniformly in $x\in[-(\mu -\lambda)t-\delta\sqrt{t},-(\mu -\lambda)t+\delta\sqrt{t}],$
\begin{equation}\label{L(f)_3}
    \lim_{t\to\infty}\frac{u_{\phi}(t,\lambda t-x)}{u_{M}(t,\lambda t-x-K_f)} = \frac{C(f,\alpha(\lambda))}{C_{M}(\alpha(\lambda))e^{\mu K_f}}.
\end{equation}
Combining \eqref{L(f)_2} and \eqref{L(f)_3}, we get
\begin{equation} \label{Eqn_C_M and C_f}
   \mu C_{M}(\alpha(\lambda))\int_{-\infty}^{\infty}\mathbb{E}[1-e^{-\int f(y+z)\mathcal{D}^{\mu}(dz) }]e^{-\mu y}dy = C(f,\alpha(\lambda)).
\end{equation}
Using \eqref{Laplace from PPP}, \eqref{Eqn_limit F-KPP integral} and \eqref{Eqn_C_M and C_f},
\begin{equation}
    \lim_{t\to\infty}\mathbb{E}\left[e^{-\langle f,\theta^\lambda_t\rangle} \right] = e^{-C(f,\alpha(\lambda))}.
\end{equation}
On the other hand, from \eqref{Laplace_DPPP} we know
\begin{equation}
    \mathbb{E}\left[e^{-\langle f, \widetilde{\mathcal E}_{\infty}^\lambda\rangle} \right] = e^{-C(f,\alpha(\lambda))}.
\end{equation}
This means
\begin{equation}
    \lim_{t\to\infty}\mathbb{E}\left[e^{-\langle f,\theta^\lambda_t\rangle} \right] = \mathbb{E}\left[e^{-\langle f, \widetilde{\mathcal E}_{\infty}^\lambda\rangle} \right],
\end{equation}
which concludes the proof.
\end{proof}
\begin{proof}[Proof of Lemma~\ref{gap_convergence}]
    Let us denote $D_t = \sum_{k = 1}^{n(t)}\delta_{\chi_k(t)-M_t}$. 
    By Corollary 7.6 in \cite{bovier2014extremal}, we get that, uniformly in $x\in[-(\mu -\lambda)t-\delta\sqrt{t},-(\mu -\lambda)t+\delta\sqrt{t}]$, conditioned on the event $\{x+M_t-\lambda t \ge -K_{f} \}$, the pair $(D_t, x+M_{t}-\lambda t +K_{f})$ jointly converges to an independent pair $(D,E_{\mu})$, where $E_{\mu}$ is an exponential random variable with rate $\mu$, and $D$ is defined as 
    \begin{equation}\label{Def_D}
        \mathbb{P}\left(D \in \cdot \right):= \lim_{t\to\infty}\mathbb{P}\left(\sum_{k = 1}^{n(t)}\delta_{\chi_k(t)-M_t}\in \cdot\mid x+M_t-\lambda t \ge -K_{f}\right).
    \end{equation}
Moreover, to identify the law of $D$, Corollary~7.6 in
\cite{bovier2014extremal} shows that the above limiting law does
not depend on the choice of $K_f$, and is the same for any
$x=-(\mu-\lambda)t+o(t)$. We may therefore take
$x=-(\mu-\lambda)t$ and $K_f=0$. Hence, if we define
\begin{equation}
\mathbb{P}\left(\mathcal D^\mu\in\cdot\right)
:=
\lim_{t\to\infty}
\mathbb{P}\left(
\sum_{k=1}^{n(t)}
\delta_{\chi_k(t)-M_t}\in\cdot
\,\middle|\,
M_t-\mu t\geq 0
\right),
\label{Def_D_mu}
\end{equation}
then
$
D\stackrel{d}{=}\mathcal D^\mu.
$
    
Therefore, we have
    \begin{equation}
    \begin{aligned}
        &\lim_{t\to\infty}\mathbb{E}[1-e^{-\sum_{k= 1}^{n(t)} f(x+\chi_{k}(t) - \lambda t)}|x+M_{t} -\lambda t\ge -K_{f}] \\
        =&\lim_{t\to\infty}\mathbb{E}[1-e^{-\int f(x+M_t-\lambda t + z)D_{t}(dz)}|x+M_{t} -\lambda t\ge -K_{f}]\\
        =&\mathbb{E}[1-e^{-\int f(E_{\mu}-K_f + z)D(dz)}].
         \end{aligned}
    \end{equation}
    Since $E_{\mu}$ and $D$ are independent and $D$ has the same distribution as $\mathcal D^{\mu}$, we get 
    \begin{equation}
        \begin{aligned}
            &\lim_{t\to\infty}\mathbb{E}[1-e^{-\sum_{k = 1}^{n(t)} f(x+\chi_{k}(t) - \lambda t)}|x+M_{t} -\lambda t\ge -K_{f}] \\
        =&\int_{0}^{\infty}\mu\mathbb{E}[1-e^{-\int f(x-K_{f}+z)\mathcal D^{\mu}(dz)}]e^{-\mu x}dx\\
        =&\int_{-K_{f}}^{\infty}\mu\mathbb{E}[1-e^{-\int f(y+z)\mathcal D^{\mu}(dz)}]e^{-\mu y}e^{-\mu K_{f}}dy.
        \end{aligned}
    \end{equation}
    Finally, the domain of the integration above can be extended to the whole real line because $f$ is supported on $[-K_f,\infty)$ and $\mathcal D^{\mu}$ is supported on $(-\infty,0]$. This completes the proof.
\end{proof}

\subsection{Intensity measure of $\widetilde{\mathcal{E}}_\infty^\lambda$ for $\lambda>\sqrt{2}$.}

It was proven in \cite{Lisa_P} that for the critical extremal point process $\widetilde{\mathcal{E}}_\infty$, the decoration $\mathcal{D}$ satisfies $\mathbb{E}[\mathcal{D}([x,0])] \sim Ce^{-\sqrt{2}x}$ as $x\to -\infty$ for some constant $C>0$. Together with \eqref{tilde-def}, this easily implies that $\widetilde{\mathcal{E}}_\infty$ has infinite intensity, i.e. $\mathbb{E}[\widetilde{\mathcal{E}}_\infty([-K, K])] = +\infty $ for all $K>0$. On the other hand, the supercritical extremal point process $\widetilde{\mathcal{E}}_\infty^\lambda$ has a finite intensity given as follows.

\begin{proposition}\label{supercritical-intensity}
Let $\lambda>\sqrt2$ and $ \mu=\lambda+\sqrt{\lambda^2-2}. $ Then, for every $f\in C_c^+(\mathbb R)$, 
 \begin{equation} \mathbb E\big[ \langle f,\widetilde{\mathcal E}^{\lambda}_{\infty}\rangle \big] = \frac{1}{\sqrt{2\pi}}\int_{\mathbb R} f(x)e^{-\mu x}\,dx . \end{equation} 
\end{proposition}
To prove this proposition, we use the following lemma, which gives an alternate representation of the Laplace transform for $\widetilde{\mathcal E}^{\lambda}_{\infty}$, that may also be of independent interest.

   \begin{lemma}\label{lemma:laplace_supercritical} For every $f\in C_c^+(\mathbb R)$, \begin{equation}\label{laplace_supercritical} \begin{aligned} &\mathbb E\big[ e^{-\langle f,\widetilde{\mathcal E}^{\lambda}_{\infty}\rangle} \big]\\
   &= \exp\Bigg\{ -\frac{1}{\sqrt{2\pi}}\Bigg[ &\int_{\mathbb R}(1-e^{-f(x)})e^{-\mu x}\,dx - \int_0^\infty\int_{\mathbb R} u_{\phi}^2(s,\lambda s-x)e^{-\mu x}\,dx\,ds \Bigg] \Bigg\}, \end{aligned} \end{equation} where $u_{\phi}$ is the solution to the F-KPP equation with the initial condition \[ u_{\phi}(0,x)=1-e^{-f(-x)}. \] \end{lemma}
   \begin{proof}[Proof of Lemma~\ref{lemma:laplace_supercritical}]
   Let $\theta_0^\lambda$ be a Poisson point process with intensity $\frac{1}{\sqrt{2\pi}}e^{-\mu x}dx$, and let $\theta_t^\lambda$ be the corresponding BBM flow with drift $\lambda$ started from $\theta_0^\lambda$ (see \eqref{theta_t-def}). 
   By the Laplace functional formula of Poisson point process,
   we have
   \[ \mathbb E\big[e^{-\langle f,\theta_t^\lambda\rangle}\big] = \exp\left\{ -\frac{1}{\sqrt{2\pi}}\int_{\mathbb R} u_{\phi}(t,\lambda t-x)e^{-\mu x}\,dx \right\}. \] 
   Recall, by Lemma~\ref{BBM-F-KPP-Connection}, 
   \[ u_{\phi}(t,\lambda t-x) = 1-\mathbb{E}[e^{-\sum_{k=1}^{n(t)} f(x + \chi_k(t) -\lambda t )}].\]
Define the family of maps $\{\pi_t\}_{t \ge 0}$ parametrized by $t$ and acting on space of all bounded measurable functions as
   \begin{equation}
       (\pi_t g)(x):= e^t\mathbb E[g(x+B_t-\lambda t)].  
   \end{equation}
   Using the strong Markov property at the first branching time and the branching property, we obtain
   \begin{equation}
   \label{mild_identity_u} u_{\phi}(t,\lambda t-x) = \pi_t(1-e^{-f(x)}) - \int_0^t \pi_{t-s} \big(u_{\phi}^2(s,\lambda s-x)\big)\,ds . 
   \end{equation}
    By \eqref{Eqn_equal expectation_app1}, for all $g \in C_c^+(\mathbb R),$
    \begin{equation}\label{Eqn_equal expectation_app0}
    \int_{\mathbb R}\pi_t( g(x))e^{-\mu x}\,dx = \int_{\mathbb R}g(x)e^{-\mu x}\,dx.
    \end{equation}
   Integrating \eqref{mild_identity_u} against $e^{-\mu x}dx$ and using \eqref{Eqn_equal expectation_app0} gives 
    \begin{equation}\label{integrated_mild_identity} \begin{aligned} \int_{\mathbb R} u_{\phi}(t,\lambda t-x)e^{-\mu x}\,dx = &\int_{\mathbb R} (1-e^{-f(x)})e^{-\mu x}\,dx \\ &- \int_0^t\int_{\mathbb R} u_{\phi}^2(s,\lambda s-x)e^{-\mu x}\,dx\,ds . \end{aligned} \end{equation} 
    Finally, by Proposition~\ref{Identify_DPPP}, $\theta_t^\lambda$ converges in distribution to $\widetilde{\mathcal E}^{\lambda}_{\infty}$. Letting $t\to\infty$ in both sides of \eqref{integrated_mild_identity} yields \eqref{laplace_supercritical}. \end{proof}
   
   \begin{proof}[Proof of Proposition~\ref{supercritical-intensity}] By Fatou's lemma and the invariance of the first moment for $\theta^\lambda_t$ (that is \eqref{Eqn_equal expectation_app1}), for every $f\in C_c^+(\mathbb R)$, 
   \begin{equation}
        \mathbb E\big[ \langle f,\widetilde{\mathcal E}^{\lambda}_{\infty}\rangle \big] \le \liminf_{t\to\infty} \mathbb E[\langle f,\theta_t^\lambda\rangle] = \frac{1}{\sqrt{2\pi}}\int_{\mathbb R}f(x)e^{-\mu x}\,dx . 
    \end{equation}
   It remains to prove the reverse inequality. Applying Lemma~\ref{lemma:laplace_supercritical} with $\varepsilon f$ in place of $f$ and using Jensen's inequality gives 
   \[ \begin{aligned} \mathbb E\big[ \langle f,\widetilde{\mathcal E}^{\lambda}_{\infty}\rangle \big] \ge \frac{1}{\sqrt{2\pi}}\left(\int_{\mathbb R} \frac{1-e^{-\varepsilon f(x)}}{\varepsilon} e^{-\mu x}\,dx - \frac1\varepsilon \int_0^\infty\int_{\mathbb R} u_{\varepsilon }^2(s,\lambda s-x)e^{-\mu x}\,dx\,ds\right) . \end{aligned} \] Here we use $u_{\varepsilon }$ to denote the solution of the F-KPP equation with initial condition $u_\varepsilon(0,x) = 1-e^{-\varepsilon f(-x)}$.
   The first term inside the bracket of the RHS above converges to $\int_{\mathbb R} f(x)e^{-\mu x}\,dx$ by dominated convergence as $\epsilon\downarrow 0$. We claim that 
   \begin{equation}\label{epsilon_error_vanishes} \lim_{\varepsilon\downarrow0} \frac1\varepsilon \int_0^\infty\int_{\mathbb R} u_{\varepsilon }^2(s,\lambda s-x)e^{-\mu x}\,dx\,ds = 0. \end{equation}
   Using the many to one lemma (Lemma \ref{Lemma many to one}) together with the simple inequality $1-e^{-x}\le x$ for $x\ge 0$,
   \begin{equation}\label{Eqn u eps bound}
       u_{\varepsilon }(s,\lambda s-x) \le 1\wedge \varepsilon \pi_s f(x).
   \end{equation}
    Moreover, if $\operatorname{supp}f\subset[-K_f,K_f]$, then \[ \pi_s f(x) \le C e^{(1-\lambda^2/2)s}e^{\lambda x}. \]
   Since $\lambda>\sqrt 2$ and $2\lambda>\mu$, we have
   \begin{equation}\label{epsilon_bound1}
   \begin{aligned}
     &\frac1\varepsilon \int_0^\infty\int_{-\infty}^0 u_{\varepsilon }^2(s,\lambda s-x)e^{-\mu x}\,dx\,ds \\
     &\le C\varepsilon \int_0^\infty e^{(2-\lambda^2)s}\,ds \int_{-\infty}^0 e^{(2\lambda-\mu)x}\,dx \xrightarrow[\varepsilon\downarrow0]{}0. 
     \end{aligned}
     \end{equation}
  For the remaining part, we use $u_{\varepsilon }^2\le \varepsilon \pi_s f$ to obtain 
   \[ 0\le \frac1\varepsilon u_{\varepsilon }^2(s,\lambda s-x)e^{-\mu x} \le e^{(1-\lambda^2/2)s} e^{-(\mu-\lambda)x}. \] 
   Thus, for all $\varepsilon>0$, $\frac{1}{\varepsilon}u_{\varepsilon }^2(s,\lambda s-x)e^{-\mu x}$ is integrable in $\mathbb R_+\times \mathbb R_+$ as $\lambda>\sqrt2$ and $\mu-\lambda=\sqrt{\lambda^2-2}>0$. 
   By \eqref{Eqn u eps bound}, $\frac{1}{\varepsilon}u_{\varepsilon }^2(s,\lambda s-x)e^{-\mu x}$ converges pointwise to $0$ as $\varepsilon \downarrow 0$. Therefore, the dominated convergence theorem gives
   \begin{equation}\label{epsilon_bound2}
         \frac1\varepsilon \int_0^\infty\int_{0}^{\infty} u_{\varepsilon }^2(s,\lambda s-x)e^{-\mu x}\,dx\,ds \xrightarrow[\varepsilon\downarrow0]{}0. 
   \end{equation}
   Combining \eqref{epsilon_bound1} with \eqref{epsilon_bound2}, we get \eqref{epsilon_error_vanishes}. Therefore \[ \mathbb E\big[ \langle f,\widetilde{\mathcal E}^{\lambda}_{\infty}\rangle \big] \ge \frac{1}{\sqrt{2\pi}}\int_{\mathbb R} f(x)e^{-\mu x}\,dx  \] 
   which concludes the proof. 
\end{proof}

 \bibliography{References}
 \bibliographystyle{alpha}
\end{document}